\documentclass[a4paper,11pt]{article}
\usepackage{titlesec, blindtext, color}
\newcommand{\hsp}{\hspace{10pt}}
\titleformat{\chapter}[hang]{\LARGE\bfseries}{\thechapter\hsp|\hsp}{0pt}{\LARGE\bfseries}
\titleformat{\section}[hang]{\large\bfseries}{\thesection\hsp|\hsp}{0pt}{\large\bfseries}
\titleformat{\subsection}[hang]{\large\bfseries}{\thesubsection\hsp|\hsp}{0pt}{\large\bfseries}
\usepackage[T1]{fontenc}
\usepackage[utf8]{inputenc}
\usepackage[english]{babel}
\usepackage[gen,right]{eurosym}
\usepackage[titles]{tocloft} 
\usepackage[style=long3col,counter=page,toc=true]{glossaries}
\usepackage[intoc,english]{nomencl}
\usepackage{%
	dsfont,
	diagbox,
	caption,
	placeins,
	float,
	verbatim,
	listings,
	hyperref,
	amssymb,amsmath,amsfonts,amsthm, mathabx,
	microtype,%
	ellipsis,%
	ae,aeguill,
	textcomp,
	url,
	longtable,
	multirow, 
	setspace, 
	graphicx,
	ziffer,
	titletoc,
	chngcntr,
	color,%
	xr,
	makeidx,%
	chngpage, 
}
\usepackage[usenames,dvipsnames]{xcolor}
\usepackage{epsfig,multicol}
\usepackage{bbm} 
\usepackage{nicefrac} 

\usepackage{overpic}   
\usepackage{wrapfig} 
\usepackage[all]{xy} 
\usepackage{pstricks} 
\usepackage{pst-math} 
\usepackage{pst-plot} 
\usepackage{tabularx} 
\usepackage{booktabs} 
\usepackage{algpseudocode} 
\usepackage{mathtools}
\usepackage{enumerate}
\usepackage{cleveref}
\usepackage{authblk}
\usepackage[round]{natbib}

\usepackage{geometry,multicol}

\newtheoremstyle{mystyle}
{}
{.5cm}
{\normalfont}
{}
{\bfseries}
{}
{1em}
{\thmname{#1}\thmnumber{ #2} \ifthenelse{\equal{#3}{}}{}{{\normalfont (\thmnote{#3})}}}

\theoremstyle{mystyle}

\newtheorem{Example}{Example}[section]
\newtheorem{Definition}[Example]{Definition}
\newtheorem{Theorem}[Example]{Theorem}
\newtheorem{Remark}[Example]{Remark}
\newtheorem{Lemma}[Example]{Lemma}
\newtheorem{Corollary}[Example]{Corollary}

\newcommand{\NN}{\mathbb{N}}

\newcommand{\RR}{\mathbb{R}}

\DeclareMathOperator{\EE}{\mathbb{E}}
\DeclareMathOperator{\PP}{\mathbb{P}}
\DeclareMathOperator{\Var}{Var}
\DeclareMathOperator{\Cov}{Cov}
\DeclareMathOperator{\Corr}{Corr}
\DeclareMathOperator{\sgn}{sgn}

\newcommand{\dd}{\mathrm{d}}

\newcommand*{\CFU}[1]{
	\ifthenelse{\equal{#1}{}}{\ensuremath{\varphi_{U}}}{\ensuremath{\varphi_{U}(#1)}}
} 

\newcommand*{\CFV}[2]{
	\ifthenelse{\equal{#1}{+}}{
		\ifthenelse{\equal{#2}{+}}{
			\ensuremath{\varphi_{UV}(s_1,s_2)}
		}
		{
			\ensuremath{\varphi_{UV}(s_1,-s_2)}
		}
	}{	
		\ifthenelse{\equal{#1}{-}}{
			\ifthenelse{\equal{#2}{+}}{
				\ensuremath{\varphi_{UV}(-s_1,s_2)}
			}
			{
				\ensuremath{\varphi_{UV}(-s_1,-s_2)}
			}
		}{
			\ifthenelse{\equal{#1} {} } {\ensuremath{\varphi_{UV}} }{\ensuremath{\varphi_{UV}(#1,#2)}}
		}	
	}
} 

\newcommand*{\CFA}[1]{
	\ifthenelse{\equal{#1}{}}{\ensuremath{\varphi}}{\ensuremath{\varphi(#1)}}
} 
\newcommand*{\CFB}[2]{
	\ifthenelse{\equal{#1}{+}}{
		\ifthenelse{\equal{#2}{+}}{
			\ensuremath{\varphi_t(s_1,s_2)}
		}
		{
			\ensuremath{\varphi_t(s_1,-s_2)}
		}
	}{	
		\ifthenelse{\equal{#1} {} } {\ensuremath{\varphi_t} }{\ensuremath{\varphi_t(#1,#2)}}
	}
} 

\newenvironment*{pf}
{\textit{Proof:}}
{\qed}

\newcommand*{\ANA}[1]{\ensuremath{{\Vert#1\Vert}^{\alpha}_{\alpha}}}
\newcommand*{\lp}[2]{\ensuremath{{\mathcal{L}}^{#1}(#2)}}
\newcommand*{\re}[1]{
	\ifthenelse{\equal{#1}{}}{
		\ensuremath{\text{Re}}	
	}{
		\ensuremath{\text{Re} \{ \psi(s_#1) \} } 
	}
}

\newcommand*{\im}[1]{
	\ifthenelse{\equal{#1}{}}{
		\ensuremath{\text{Im}}
		
	}{
		\ensuremath{\text{Im} \{ \psi(s_#1) \} } 
	}
}

\title{Long Range Dependence for  Stable Random Processes}
\author[1]{Vitalii Makogin} 
\author[2]{Marco Oesting}
\author[1, $\star$]{Albert Rapp}
\author[1]{Evgeny Spodarev}

\affil[1]{{\footnotesize Ulm University, E-mails: vitalii.makogin@uni-ulm.de, albert.rapp@uni-ulm.de, evgeny.spodarev@uni-ulm.de}}
\affil[2]{{\footnotesize University of Siegen, E-Mail: oesting@mathematik.uni-siegen.de}}
\affil[$\star$]{\footnotesize corresponding author}

\date{\today}

\definecolor{Marco}{rgb}{1,0.49,0}
\definecolor{Albert}{rgb}{0.11, 0.67, 0.84}
\definecolor{Vitalii}{rgb}{1, 0.0, 0.13}
\definecolor{Evgeny}{rgb}{0.25, 0.41, 0.88}

\begin{document}
	
	\maketitle
	
	\begin{abstract}
		We investigate long and short memory in $\alpha$-stable moving averages and max-stable processes with $\alpha$-Fr\'echet marginal distributions.
		As these processes are heavy-tailed, we rely on the notion of long range dependence suggested by \cite{Kulik_Spodarev_LRD} based on the covariance of excursions. Sufficient conditions for the long and short range dependence of $\alpha$-stable moving averages are proven in terms of integrability of the corresponding kernel functions. For max-stable processes, the extremal coefficient function is used to state a necessary and sufficient condition for long range dependence.
		
		AMS Subj. Class.: 60G10, 60G52, 60G70.
	\end{abstract}
	
	
	\allowdisplaybreaks
	\section{Introduction}
	\label{sec: LRD}

	The occurrence of long memory in time series has been known for a long time starting from the work of \cite{hurst51}. Since then, this phenomenon 
		has been observed and studied in applications in various fields including 
		biophysical data (\cite{burnecki2012farima}), network traffics (\cite{Pilip2016}), neuroscience (\cite{botcharova2014markers}), and geosciences (\cite{Montillet2015}), etc. 
		A typical example in financial applications (see e.g. \cite{cheung1995search} and \cite{panas2001estimating}) is a stationary solution of a autoregressive moving average FARIMA($p,d,q$) process with $\alpha-$stable innovations.
		In light of the variety of applications, a wide range of statistical models and methods for long range dependent processes has been developed, see, for instance, \cite{Avram1986}, \cite{Kasahara1988}, \cite{Kokoszka1996}  for classical ones, and \cite{Magdziarz2007} \cite{Beran2012}, \cite{Jach2012}, \cite{Koul2018} for more recent developments.  For
		a broader overview, we recommend the books of \cite{Doukhan2003}, \cite{beran:kulik:2013}, and \cite{Samorod16}.
		These instruments rely on the explicit definition of long range dependence (LRD, for short) of a stationary time series or, more generally, a stationary stochastic process  $X = \{X(t), \, t \in T\}$. Here and throughout this paper, stationarity is understood in the sense that all finite-dimensional distributions of $X$ are invariant
	under translations. There are many definitions of LRD in the literature depending on the class of processes to which $X$ belongs. For instance, if $X$ has a finite variance the following definition is classical, cf. \cite[P. 194-195]{Samorod16}:   
	\begin{Definition}
		\label{Def: LRD_classical}
		A stationary stochastic process $X = \{ X(t), \, t \in T \}$ on some domain $T \subset \mathbb{R}$ with $\mathbb{E}\big[\vert X(0)\vert^2\big] < \infty$ is called {\it long range dependent} if 
		\begin{align*}
		\int_{T} \vert C(t) \vert \, \dd t = \infty, 
		\end{align*}
		where $C(t) = \Cov(X(0), X(t))$, $t \in T$, is its covariance function. For processes in discrete time, the integral above should be changed to a sum.
		
		Also, $X$ is {\it antipersistent}  if $\int_{T} \vert C(t) \vert \, \dd t <\infty$, $\int_{T}  C(t)  \, \dd t = 0$, and {\it short range dependent}, otherwise.
	\end{Definition}
	Alternative  definitions of long memory rely e.g. on  the unboundedness  of the spectral density of $X$ at zero, growth comparisons of partial sums, phase transition in limit theorems for sums or maxima, etc., cf. \cite{HeydeYang97,DehlPhil02,Samorod04,Lavancier06,GiraitisKoulSurg12,beran:kulik:2013,Paul16,Samorod16, Jach2012}. 
	
	Many of these approaches fail for heavy-tailed stochastic processes whose variance does not exist. Such processes occur, for instance, in modelling of network data, in finance and in insurance (see e.g. \cite{Kokoszka1997} who call the FARIMA($p,d,q$) process with  $\alpha$-stable innovations long range dependent  if $d \in (0,1-1/\alpha)$ or \cite{EKM97, resnick-2007}).  In order to allow for the analysis of long memory behaviour in a broader setting, \cite{Kulik_Spodarev_LRD} propose to consider the covariance of indicator functions of excursions and introduce
	\begin{Definition}[]
		\label{def: LRD_Spodarev}
		A real-valued stationary stochastic process $X = \{ X(t), \, t \in T \}$ where $T$ is an unbounded subset of $\mathbb{R}$ is short range dependent (SRD) if
		\begin{align}
		\int_{T} &\int_{\mathbb{R}} \int_{\mathbb{R}} \Big\vert \Cov (\mathds{1} \{ X(0) > u \}, \mathds{1} \{ X(t) > v \})\Big\vert \, \mu(\dd u) \, \mu(\dd v)\, \dd t < \infty \label{eq: LRD_Spodarev}
		\end{align}
		for any finite measure $\mu$ on $\mathbb{R}$. Otherwise, i.e. if there exists a finite measure $\mu$ such that the integral in inequality $\eqref{eq: LRD_Spodarev}$ is infinite, $X$ is long range dependent. For  stochastic processes in discrete time, the integral $\int_T\, \dd t$ should be replaced by the summation $\sum_{t \in T: \, t\neq 0}$.
	\end{Definition}
	
	One major advantage of this  definition is that the above covariance exists in any case due to the boundedness of the indicators. Furthermore, the definition turns out to be useful as it offers the applicability of limit theorems for certain functionals of the process of interest.
	
	In practice, however, the computation of the multiple integral in \eqref{eq: LRD_Spodarev} might prove to be tricky. Therefore, we restrict ourselves here to the wide class of positively associated stochastic processes, including the class of infinitely divisible moving average processes with nonnegative kernels \cite[Chapter 1, Theorem 3.27]{Bulinski_Shashkin}. This will allow us to eliminate the absolute value in \eqref{eq: LRD_Spodarev}.
	
	To introduce the notion of positive association, we need the class $\mathcal{M}(n)$ of real-valued bounded coordinate-wise nondecreasing Borel functions on $\mathbb{R}^n$, $n \in \mathbb{N}$. For a real-valued stochastic process $X = \{ X(t), \, t \in T \}$ and a set $I \subset T$, we denote $X_I = \{ X(t), \, t \in I \}$. \
	\begin{Definition}[]
		\label{def: association}
		A real-valued stochastic process $X = \{ X(t), \, t \in T \}$  is {\it positively associated} if $\Cov(f(X_I), g(X_J)) \geq 0$	for any disjoint finite subsets $I,J \subset T$ and all functions $f \in \mathcal{M}(\vert I \vert)$ and $g \in \mathcal{M}(\vert J \vert)$.
	\end{Definition}
	
	By setting $I = \{ 0 \}$ and $J = \{ t \}$ for $t \neq 0$, $f(x) = \mathds{1}\{ x > u \}$ and $g(x) = \mathds{1}\{ x > v \}$ for $u,v \in \RR$, we have $f \in \mathcal{M}(\vert I \vert)$ and $g \in \mathcal{M}(\vert J \vert)$. Consequently, for a positively associated stochastic process $X$, it holds $\Cov (\mathds{1} \{ X(0) > u \}, \mathds{1} \{ X(t) > v \}) = \Cov(f(X_I), g(X_J))\ge 0$, i.e.\ the absolute value in \eqref{eq: LRD_Spodarev} can be omitted.
	\medskip
	
	In this paper, we consider two important subclasses of positively associated stationary processes that satisfy certain stability properties. More precisely, we study $\alpha$-stable moving averages and max--stable processes with $\alpha$-Fr\'echet marginals. As these processes are heavy-tailed, the classical definition of LRD (\Cref{Def: LRD_classical})
	does not apply. Instead, we check \Cref{def: LRD_Spodarev}.

	With regard to this endeavor, we first establish a general framework to compute the double integral $\int_{\mathbb{R}} \int_{\mathbb{R}} \Cov (\mathds{1} \{ X(0) > u \}, \mathds{1} \{ X(t) > v \})\, \mu(\dd u)\, \mu(\dd v)$ by inverting the univariate characteristic function $\CFA{s}$ of $X(0)$ and the bivariate characteristic function $\CFB{+}{+}$ of $(X(0), X(t))$. Thus, our \Cref{thm: SRD and LRD} yields
	\begin{align*}
	&\int_{\mathbb{R}} \int_{\mathbb{R}} \Cov (\mathds{1} \{ X(0) > u \}, \mathds{1} \{ X(t) > v \})\, \mu(\dd u)\, \mu(\dd v) \\[2mm]
	&= \frac{1}{2 \pi^2} \int_{\mathbb{R}_+} \int_{\mathbb{R}_+} \frac{1}{s_1s_2} \re{} \bigg\{ \big( \CFB{+}{-} - \CFA{s_1}\CFA{-s_2} \big) \overline{\psi(s_1)}\psi(s_2) \bigg\}  \\[2mm]
	&\hspace{1cm} - \frac{1}{s_1s_2} \re{} \bigg\{ \big( \CFB{+}{+} - \CFA{s_1}\CFA{s_2} \big) \overline{\psi(s_1)\psi(s_2)} \bigg\} \, \dd s_1 \, \dd s_2
	\end{align*}
	where $\psi(s) = \int_{\mathbb{R}} \exp \{ isx \}\, \mu(dx)$ is the Fourier transform of measure $\mu$. \\
	Integrating this relation with respect to $t$ will establish short or long range dependence according to  \Cref{def: LRD_Spodarev}. Subsequently, we will apply this result to get the LRD of {\it symmetric $\alpha$-stable (S$\alpha$S) moving averages} which are defined as follows. 
	\begin{Definition}[\cite{SamorodTaqqu94}]
		\label{def: SaS MA}
		Let $m$ be a measurable function with $m \in \lp{\alpha}{\mathbb{R}}$, $\alpha \in (0,2)$.  Then, a S$\alpha$S moving average process with parameter $\alpha \in (0,2)$ and kernel function $m$ is a stochastic process $X = \{ X(t), \, t \in \mathbb{R} \}$ defined by
		\begin{align}
		X(t) = \int_{\mathbb{R}} m(t-x)\ \Lambda(dx), \quad t \in \mathbb{R}, \label{eq: SASMA}
		\end{align} 
		where $\Lambda$ is a S$\alpha$S random measure with Lebesgue control measure.
	\end{Definition}
	Here and throughout the paper, we use the notation $m \in \lp{p}{A}$, $p>0$, to imply that $ \int_A \vert m(x)\vert ^p\, \dd x < \infty$.

	Regarding the SRD/LRD of the process $X$ given in \eqref{eq: SASMA}, our main result relies on the notion of $\alpha$--spectral covariance $\rho_t = \int_{\mathbb{R}} (m(-x)m(t-x))^{\alpha/2} \, \dd x$, $t \in \mathbb{R}$,  where $m(x)\ge 0$, $x\in\mathbb{R}$. The  $\alpha$--spectral covariance was first introduced by \cite{Paulauskas1976} and its properties were studied in \cite{Damarackas2014} and \cite{Damarackas2017}. In  \cite{Paul16}, it was discussed how the integrability of $\rho_t$ can be used for the definition of the memory property. Here, we establish by \Cref{thm: SRD_rho} that $X$ is short range dependent if $\rho_t \in \lp{1}{\mathbb{R}}$ or, equivalently, $m \in \lp{\alpha/2}{\mathbb{R}}$. Also, \Cref{thm: LRD_min_m} establishes long range dependence if $\int_{\mathbb{R}} \int_{\mathbb{R}} (m^\alpha(x) \wedge m^\alpha(t)) \, \dd x \, \dd t = \infty$ where $a\wedge b$ is the minimum of $a$ and $b$. These results hold also for  $\alpha$-stable linear time series if integrals are replaced by sums. 
	
	To put our results into context, one may refer to other research and discussion on memory properties of $\alpha-$stable processes such as \cite{Rachev2002}, \cite{Maejima2003}, \cite{Samorod04}.
	Also, we demonstrate how our findings are meaningful in practice by detecting LRD in a real world data set consisting of daily log-returns based on the opening price of the Intel corporation share.
	
	\medskip
	
	Analogously to $\alpha$-stable processes, which have become popular as limits of rescaled sums of stochastic processes, \emph{max-stable processes} have become a widely used concept in extreme value analysis occurring as limiting models for maxima. Thus, they have found applications in various areas such as meteorology \citep[see e.g.]{coles93, buishand-etal-08, davison-gholamrezaee-2012, oesting-friederichs-schlather-2017}, hydrology \citet{asadi-davison-engelke-18} and finance \citep{zhang-smith-2010}. Max-stable processes are defined as follows.
	
	\begin{Definition}
		A real-valued stochastic process $X=\{X(t), \ t \in T\}$ is called a {\it max-stable process} if, for all $n \in \NN$, there exist functions $a_n: T \to (0,\infty)$ and $b_n : T \to \RR$ such that
		$$ \left\{ \max_{i=1}^n \frac{X_i(t) - b_n(t)}{a_n(t)}, \, t \in T \right\}
		\stackrel{d}{=} \{X(t), \, t \in T \}, $$
		where the processes $X_i$, $i \in \NN$, are independent copies of $X$, and $\stackrel{d}{=}$ means equality in distribution. 
		If the index set $T$ is finite, $X$ is also called a {\it max-stable vector}. 
	\end{Definition}	
	
	It follows from the univariate extreme value theory that the marginal 
	distributions of a max-stable process are either degenerate or follow a
	Fr\'echet, Gumbel or Weibull law. While covariances always exist in the Gumbel and Weibull case and, thus, the classical notion of long-range dependence applies, we will consider the case when $X$ is a stationary max-stable process with $\alpha$-Fr\'echet marginal distributions, i.e.\ $\PP(X(t) \leq x) = \exp(-x^{-\alpha})$ for all $x > 0$ and some $\alpha > 0$ and all $t \in T$. Here, covariances do not exist if $\alpha \leq 2$.

	In combination with Definition \ref{def: LRD_Spodarev}, a well-established dependence measure for max--stable stochastic processes allows for an easily tractable condition for short and long memory, respectively.
	More specifically, we use the pairwise extremal coefficient $\{\theta_t, \ t \in \RR\}$ defined via the relation $\PP(X(0) \leq x, \ X(t) \leq x) = \PP(X(0) \leq x)^{\theta_t}$, which holds for all $x > 0$, to show  that a stationary max-stable process with $\alpha$-Fréchet marginal distributions is long range dependent if and only if $\int_{\mathbb{R}} (2-\theta_t) \, \dd t = \infty$ (cf.~\Cref{prop:lrd}).
	\medskip
	
	To summarize, our paper is structured as follows: \Cref{sec: Vitalii} establishes the framework to invert the bivariate characteristic functions.
	In \Cref{sec: AlphaStable}, we make use of this framework to find conditions for long range dependence of symmetric $\alpha$-stable moving averages and linear time series, while, in \Cref{sec: MaxStable}, we investigate long range dependence of a stationary max-stable process with $\alpha$-Fréchet marginals. Finally, we model the daily log-returns of an Intel corporation share by a S$\alpha$S moving average $X$ and show that $X$ is LRD in \Cref{sec: data}. 
	For the sake of legibility, some of the proofs have been left out of the main part of this paper. They can be found in the Appendix.
	
	
	\section{From Characteristic Function to Covariance of Indicators}
	\label{sec: Vitalii}
	
	In this section, we express  the covariance of indicators of excursions  of  random variables above some levels $u,v$ through their uni- and bivariate characteristic functions. Notice that for random variables $U$ and $V$ it holds that
	\begin{align}
	\Cov (\mathds{1} \{ U > u \}, \mathds{1} \{ V > v \}) = \mathbb{P}(U \leq u, V \leq v) - \mathbb{P}(U \leq u)\mathbb{P}(V \leq v). \label{lem: Covariance with joint and marginal distr}
	\end{align}
	
	\allowdisplaybreaks
	
	%
	%

	\begin{Theorem}[]
		\label{thm: Integrated Covariance}
		Suppose $U$ and $V$ are identically distributed random variables with marginal characteristic function $\CFU{}$ and joint characteristic function  $\CFV{}{}$. Then, for a finite measure $\mu$ with its Fourier transform denoted by $\psi: \mathbb{R} \rightarrow \mathbb{C}, \psi(s) = \int_{\mathbb{R}} \exp \{ isx \}\, \mu(dx)$ it holds that
		\begin{align}
		&\int_{\mathbb{R}} \int_{\mathbb{R}} \Cov (\mathds{1} \{ U > u \}, \mathds{1} \{ V > v \})\, \mu(\dd u)\, \mu(\dd v) \notag \\[2mm]
		&=\frac{1}{4 \pi^2} \int_{\mathbb{R}} \int_{\mathbb{R}} \frac{1}{s_1 s_2} \Big(  \varphi_U(s_1)\varphi_U(s_2) - \CFV{+}{+} \Big)  \overline{\psi(s_1)\psi(s_2)}\  ds_1\, ds_2 \notag \\[2mm]
		&= \frac{1}{2 \pi^2} \int_{\mathbb{R}_+} \int_{\mathbb{R}_+} \bigg[ \frac{1}{s_1s_2} \re{} \bigg\{ \big( \CFV{+}{-} - \CFU{s_1}\CFU{-s_2} \big) \overline{\psi(s_1)}\psi(s_2) \bigg\} \notag \\[2mm]
		&\hspace{2cm} - \frac{1}{s_1s_2} \re{} \bigg\{ \big( \CFV{+}{+} - \CFU{s_1}\CFU{s_2} \big) \overline{\psi(s_1)\psi(s_2)} \bigg\} \bigg] \, \dd s_1 \, \dd s_2. \label{eq: integrated covariance}
		\end{align}
		
		\begin{pf}
Let $U^{\prime}$ and $V^{\prime}$ be independent copies of $U$ and $V$. Then
			\begin{align}
			&\int_{\mathbb{R}} \int_{\mathbb{R}} \Cov (\mathds{1} \{ U > u \}, \mathds{1} \{ V > v \})\, \mu(\dd u)\, \mu(\dd v) \notag \\
			&= \int_{\mathbb{R}} \int_{\mathbb{R}} \mathbb{E} \big[\mathds{1} \{ U > u, V > v \} - \mathds{1} \{ U^{\prime} > u, V^{\prime} > v \}\big]\, \mu(\dd u)\, \mu(\dd v) \notag\\
			&= \lim_{a \to \infty} \mathbb{E} \int_{\mathbb{R}} \int_{\mathbb{R}} \big[\mathds{1} \big\{ U > u > -a, V > v> -a \big\} - \mathds{1} \big\{ U^{\prime} > u > -a, V^{\prime} > v > -a \big\}\big]\, \mu(\dd u)\, \mu(\dd v). \label{eq: before_plancherel}
			\end{align}
			If we denote the difference of the two indicators by $f_a(u, v)$, then by \cite[Theorem 19.12]{Schilling_measure_theory}\footnote{We thank René Schilling for his idea which simplifies our original proof.} we get that the last equality in \eqref{eq: before_plancherel} simplifies to
			\begin{align}
			&\lim_{a \to \infty} \frac{1}{4 \pi^2} \mathbb{E} \int_{\mathbb{R}} \int_{\mathbb{R}} \widehat{f}_a(s_1, s_2) \overline{\psi(s_1)\psi(s_2)}\  ds_1\, ds_2, \label{eq: after_Schilling_thm}
			\end{align} 
			where $\widehat{f}_a(s_1, s_2) = \int_{\RR} \int_{\mathbb{R}} e^{i(s_1u+s_2v)} f_a(u,v)\ du\, dv$. By \Cref{lem: Schilling_lemma} one can interchange the expectation and the integrals in equation \eqref{eq: after_Schilling_thm} and computes
			\begin{align}
			\mathbb{E} \widehat{f}_a(s_1, s_2) = \frac{1}{s_1 s_2} \Big(  \CFU{s_1}\CFU{s_2} -\CFV{+}{+} \Big)
			\end{align}
			which is independent of $a$. Thus, equation \eqref{eq: after_Schilling_thm} simplifies to
			\begin{align*}
			\frac{1}{4 \pi^2} \int_{\mathbb{R}} \int_{\mathbb{R}} \frac{1}{s_1 s_2} \Big(  \CFU{s_1}\CFU{s_2} - \CFV{+}{+} \Big) \overline{\psi(s_1)\psi(s_2)}\  ds_1\, ds_2.
			\end{align*}
			The second identity in \eqref{eq: integrated covariance} follows from splitting the integrals into the positive and negative half-lines and substituting afterwards.
		\end{pf}
	\end{Theorem}
	
	\begin{Corollary}
		\label{Corr: SRD and LRD symmetric}
		Under the assumptions of \Cref{thm: Integrated Covariance}, suppose that the random vector $(U,V)$ is symmetric. Then relation \eqref{eq: integrated covariance} simplifies to
		\begin{align}
		&\int_{\mathbb{R}} \int_{\mathbb{R}} \Cov (\mathds{1} \{ U > u \}, \mathds{1} \{ V > v \})\, \mu(\dd u)\, \mu(\dd v) \notag\\[2mm]
		&= \frac{1}{2 \pi^2} \int_{\mathbb{R}_+} \int_{\mathbb{R}_+} \bigg[ \frac{1}{s_1s_2}  \big( \CFV{+}{-} - \CFU{s_1}\CFU{-s_2} \big) \re{} \big\{\overline{\psi(s_1)}\psi(s_2) \big\} \notag\\[2mm]
		&\hspace{2cm} - \frac{1}{s_1s_2}  \big( \CFV{+}{+} - \CFU{s_1}\CFU{s_2} \big) \re{} \big\{\psi(s_1)\psi(s_2) \big\} \bigg] \, \dd s_1 \, \dd s_2 \label{eq: thm_sym_case1}\\[2mm]
		&= \frac{1}{2 \pi^2} \int_{\mathbb{R}_+} \int_{\mathbb{R}_+} \frac{1}{s_1s_2}  \Big( \CFV{+}{-} - \CFV{+}{+} \Big) \re{1}\re{2} \notag \\[2mm]
		&\hspace{0.5cm} + \frac{1}{s_1s_2}  \Big(\CFV{+}{-}+ \CFV{+}{+} - 2\CFU{s_1}\CFU{s_2} \Big) \im{1} \im{2} \, \dd s_1 \, \dd s_2. \label{eq: thm_sym_case2}
		\end{align}
	\end{Corollary}	
	\begin{proof}
		Equality \eqref{eq: thm_sym_case1} follows immediately from $\CFU{}$ and $\CFV{}{}$ being real-valued as characteristic functions of a symmetric random variable and random vector, respectively. 
		
		Equality \eqref{eq: thm_sym_case2} follows from	$\text{Re}\{ xy \} = \text{Re}\{ x \}\text{Re}\{ y \} - \text{Im}\{ x \}\text{Im}\{ y \}$ for any $x, y \in \mathbb{C}$.
	\end{proof}

	If the stationary real-valued stochastic process $X = \{ X(t), t \in \mathbb{R} \}$ is positively associated, we can apply \Cref{thm: Integrated Covariance} and, in the symmetric case, \Cref{Corr: SRD and LRD symmetric} to $X(0)$ and $X(t)$ to check the long range dependence of $X$. \\	
    To do so, let $T=\mathbb{R}$ in integral \eqref{eq: LRD_Spodarev}. However, the resulting expressions in \eqref{eq: integrated covariance}, \eqref{eq: thm_sym_case1} or \eqref{eq: thm_sym_case2} might prove difficult to integrate w.r.t.\ $t$ over the whole real line. Thus, it is worth noting that the following lemma allows us to restrict integration to unbounded subsets over which it might be easier to integrate.
	
	\begin{Lemma}[]
		\label{lemma: Restriction of X}
		Let $|\cdot|$ denote the Lebesgue measure on $\mathbb{R}$ and let $A \subset \RR$ be an arbitrary subset with $|A^c| < \infty$. Then, a process $X = \{ X(t), \, t \in \mathbb{R} \}$ is SRD or LRD iff $X_A = \{ X(t), \, t \in A \}$ is SRD or LRD, respectively.
		
		\begin{pf}
			We split up the integral in relation \eqref{eq: LRD_Spodarev} into $A$ and $A^c$
			\begin{align*}
			&\int_{\mathbb{R}} \int_{\mathbb{R}}  \int_{\mathbb{R}} \bigg\vert \text{Cov}\bigg(\mathds{1}\big\{ X(0) > u \big\}, \mathds{1}\big\{ X(t) > v \big\}\bigg) \bigg\vert \, \mu(\dd u) \, \mu(\dd v) \, \dd t \\[2mm]
			=&\int_{A} \int_{\mathbb{R}}  \int_{\mathbb{R}} \bigg\vert \text{Cov}\bigg(\mathds{1}\big\{ X(0) > u \big\}, \mathds{1}\big\{ X(t) > v \big\}\bigg) \bigg\vert \, \mu(\dd u) \, \mu(\dd v) \, \dd t \\[2mm]
			\vspace{1cm}&+
			\int_{A^c} \int_{\mathbb{R}}  \int_{\mathbb{R}} \underbrace{\bigg\vert \text{Cov}\bigg(\mathds{1}\big\{ X(0) > u \big\}, \mathds{1}\big\{ X(t) > v \big\}\bigg) \bigg\vert}_{\leq 1} \, \mu(\dd u) \, \mu(\dd v) \, \dd t.
			\end{align*}
			As the integral over $A^c$ is finite in any case, the integral in relation \eqref{eq: LRD_Spodarev} is finite iff $X_A$ is SRD. 
		\end{pf}
	\end{Lemma}
	
	Now we give the main result of this section showing the use of characteristic functions to check the short or long range dependence of $X$. 
	
	\begin{Theorem}[]
		\label{thm: SRD and LRD}
		Suppose we have a stationary real-valued, positively associated stochastic process $X = \{ X(t),\, t \in \mathbb{R} \}$ with absolutely continuous marginal distributions. Denote the univariate characteristic function of $X(0)$ by $\CFA{}$ and the bivariate characteristic function of $(X(0), X(t))$ by $\CFB{}{}$. Furthermore, let $A \subset \RR$ be an arbitrary subset with $|A^c| < \infty$.
		\begin{enumerate}[(a)]
			\item  Then, $X$ is short range dependent if
			\begin{align}
			\int_{A} &\int_{\mathbb{R}} \int_{\mathbb{R}} \Cov (\mathds{1} \{ X(0) > u \}, \mathds{1} \{ X(t) > v \}) \, \mu(\dd u) \, \mu(\dd v)\, \dd t \notag \\[2mm]
			&= \frac{1}{2 \pi^2} \int_{A} \int_{\mathbb{R}_+} \int_{\mathbb{R}_+} \bigg[ \frac{1}{s_1s_2} \re{} \bigg\{ \big( \CFB{+}{-} - \CFA{s_1}\CFA{-s_2} \big) \overline{\psi(s_1)}\psi(s_2) \bigg\} \notag \\[2mm]
			&\hspace{1cm} - \frac{1}{s_1s_2} \re{} \bigg\{ \big( \CFB{+}{+} - \CFA{s_1}\CFA{s_2} \big) \overline{\psi(s_1)\psi(s_2)} \bigg\} \bigg] \, \dd s_1 \, \dd s_2\, \dd t < \infty \label{eq: SRD_or_LRD}
			\end{align}
			for any finite measure $\mu$ with Fourier transform  $\psi(s) = \int_{\mathbb{R}} \exp \{ isx \}\, \mu(dx)$.
			
			\item Additionally, if $(X(0), X(t))$ is symmetric for all $t \in \mathbb{R}$, then condition \eqref{eq: SRD_or_LRD} rewrites as
			\begin{align}
			&\int_{A}\int_{\mathbb{R}} \int_{\mathbb{R}} \Cov (\mathds{1} \{ X(0) > u \}, \mathds{1} \{ X(t) > v \})\, \mu(\dd u)\, \mu(\dd v)\, \dd t \notag \\[2mm]
			&= \frac{1}{2 \pi^2} \int_{A} \int_{\mathbb{R}_+} \int_{\mathbb{R}_+}  \bigg[ \frac{\CFB{+}{-} - \CFB{+}{+}}{s_1s_2}  \re{1}\re{2}\notag  \\[2mm]
			& -   \frac{\CFB{+}{-} + \CFB{+}{+} - 2\CFA{s_1}\CFA{s_2}}{s_1s_2}  \im{1} \im{2} \bigg] \, \dd s_1 \, \dd s_2\, \dd t <\infty \label{eq: SRD_or_LRD2}.
			\end{align} 
		\end{enumerate}
		Otherwise, i.e.\ if there exists a finite measure $\mu$ such that the integral in \eqref{eq: SRD_or_LRD} is infinite, $X$ is long range dependent.
		
		\begin{pf}
			\begin{enumerate}[(a)]
				\item  Take $U = X(0)$ and $V=X(t),$ where $t \in \mathbb{R}$ in \Cref{thm: Integrated Covariance}. 
				Then, $U$ and $V$ are absolutely continuous and identically distributed random variables. Therefore, the equality in \eqref{eq: SRD_or_LRD} is established by \Cref{thm: Integrated Covariance}. It follows by relation \eqref{eq: SRD_or_LRD} and \Cref{lemma: Restriction of X} that $X$ is SRD. Similarly, $X$ is LRD if \eqref{eq: SRD_or_LRD} is infinite.  
				\item Follows analogously by using \Cref{Corr: SRD and LRD symmetric}.
			\end{enumerate}
		\end{pf}
	\end{Theorem}
	
	
	\section{Long Range Dependence of $\alpha$--stable Moving Averages}
	\label{sec: AlphaStable}
	In this section, we investigate the LRD of S$\alpha$S moving averages in continuous and discrete time.
	
	By \Cref{def: SaS MA}, a symmetric $\alpha$-stable moving average with kernel function $m \in \lp{\alpha}{\mathbb{R}}$, $\alpha < 2$, is defined by $X(t) = \int_{\RR} m(t-x) \ \Lambda(\dd x)$, $t \in \mathbb{R}$, where $\Lambda$ is a symmetric $\alpha$-stable random measure.
	\begin{Remark}[]
		\begin{enumerate}[(a)]
			\label{rem: ID moving averages}
			
			\item 	Note that the S$\alpha$S moving average process $X = \{ X(t), \, t \in \mathbb{R} \}$ is stationary, $X(0)$ is absolutely continuous and, by Property 3.2.1 from \cite{SamorodTaqqu94}, the random vector $(X(0),X(t))$ is symmetric for every $t \in \mathbb{R}$.
			
			\item By \cite{Bulinski_Shashkin}, Theorems 1.3.5 and 1.3.27, $X$ is positively associated if the kernel function $m$ is nonnegative. 
			
			\item To exclude the trivial case $X(t) = 0$ for all $t \in \RR$ we always assume that the Lebesgue measure of the set $\{ x\in \mathbb{R} \, \vert \; m(x) > 0 \}$ is positive.	
		\end{enumerate}
	\end{Remark}
	
	By \cite{SamorodTaqqu94}, Proposition 3.4.2., the characteristic function $\varphi:\mathbb{R}\to \mathbb{C}$ of $X(t)$, $t \in \mathbb{R}$, is given by 
	\begin{align}
	\CFA{s}=\exp \bigg\{ - \vert s \vert ^ \alpha \int_{\mathbb{R}} \vert  m(x) \vert ^ \alpha \, \dd x \bigg\}, \quad  s\in \mathbb{R}. \label{eq: char_fct_univ}
	\end{align} 
	Moreover, the bivariate characteristic function $\varphi_t:\mathbb{R}\times\mathbb{R} \to \mathbb{C}$ of $\big( X(0),X(t) \big),$ $t\in \mathbb{R}$ is given by
	\begin{align}
 	\CFB{+}{+} =\exp \bigg\{ - \int_{\mathbb{R}} \vert s_1m(-x)+s_2m(t-x)\vert^{\alpha} \, \dd x \bigg\}, \quad s_1,s_2\in \mathbb{R}. \label{eq: char_fct_biv}
	\end{align}

	Before we get to our main result, we need to introduce the \textit{$\alpha$-spectral covariance} of a stable vector as defined by \cite[equation (11)]{Damarackas2017}. Let $\mathbb{S}^1=\{\mathbf{x}\in \mathbb{R}^2: \|\mathbf{x}\|=1\}$ be the unit circle.  Recall that a random vector $Z = (X_1, X_2)$ is symmetric $\alpha$-stable with parameter $\alpha$ if there exists a finite measure $\Gamma$ on $\mathbb{S}^1$, the so-called \textit{spectral measure}, such that the characteristic function of $Z$ is given by
	\begin{align*}
	\mathbb{E}e^{i\langle s, Z \rangle} = \exp \bigg\{ -\int_{\mathbb{S}^1} \vert \langle s,x \rangle \vert ^{\alpha} \Gamma(dx) \bigg\}, s \in \mathbb{R}^2, 
	\end{align*}
	where $\langle \cdot\, , \cdot \rangle$ is the standard inner product on $\mathbb{R}^2$.
	
	\begin{Definition}[]
		Suppose $(X_1,X_2)$ is an $\alpha$-stable random vector with spectral measure $\Gamma$, then the { \it $\alpha$-spectral covariance} of $X_1$ and $X_2$ is given by
		\begin{align}
		\rho = \int_{\mathbb{S}^1} |s_1 s_2|^{\alpha/2} \sgn(s_1 s_2)\ \Gamma(d(s_1,s_2)). \label{eq: alpha-convolution}
		\end{align}
	\end{Definition}
	
	Let us calculate the $\alpha$-spectral covariance of $(X(0), X(t))$, $t \in \mathbb{R}$, where $X$ is a S$\alpha$S moving average.
	
	\begin{Lemma}[]
		\label{thm: alpha_convolution_SASMA_ASLP}
		Suppose $X = \{ X(t) , t \in \mathbb{R}\}$ with $X(t) = \int_{\RR} m(t-x)\ \Lambda(dx)$ is a S$\alpha$S moving average process. Then, the $\alpha$-spectral covariance of $(X(0), X(t))$, $t \in \mathbb{R}$, is given by 
		\begin{align}
		\rho_t = \int_{\mathbb{R}} m^{\alpha/2}(-x)m^{\alpha/2}(t-x)\sgn(m(-x)m(t-x))\, \dd x. \label{eq: rho_t}
		\end{align}
		
		\begin{pf}
			Denote $m_1(x) = m(-x)$ and $m_2(x)=m(t-x)$, Proposition 3.4.3 in \cite{SamorodTaqqu94} and the the symmetry of $\Lambda$ yields that $(X(0), X(t))$ is S$\alpha$S with spectral measure $\Gamma$ defined for all Borel sets $A \subset \mathbb{S}^1$ by 
			\begin{align*}
			\Gamma(A) &= \frac{1}{2} \int_{g^{-1}(A)} \Big( m_1^2(x) + m_2^2(x) \Big)^{\alpha/2}\, \dd x  
			+ \frac{1}{2} \int_{g^{-1}(-A)} \Big( m_1^2(x) + m_2^2(x) \Big)^{\alpha/2}\, \dd x\\[2mm]
			&=\gamma(g^{-1}(A)) + \gamma(g^{-1}(-A)) =:(\gamma \circ g^{-1})(A) + (\gamma \circ g^{-1})(-A),
			\end{align*} 
			where
			\begin{align*}
			g(x) = \left(
			\frac{m_1(x)}{\big( m_1^2(x)+m_2^2(x) \big)^{1/2}}\ ,  \ \frac{m_2(x)}{\big( m_1^2(x)+m_2^2(x) \big)^{1/2}}
			\right),\quad x\in \mathbb{R}.
			\end{align*}
			Hereby $\gamma$ is an absolutely continuous measure w.r.t.~the Lebesgue measure with density $\frac{1}{2}( m_1^2(x)+m_2^2(x))^{\alpha/2}$. With $f(s_1, s_2) =  |s_1 s_2|^{\alpha/2} \sgn(s_1 s_2)$ we get
			\begin{align*}
			\int_{\mathbb{S}^1} f\ d(\gamma \circ g^{-1}) &= \int_{g^{-1}(\mathbb{S}^1)} f \circ g\ d\gamma = \int_{\RR} \frac{m_1^{\alpha/2}(x)\, m_2^{\alpha/2}(x)\sgn(m_1(x)m_2(x))}{\big( m_1^2(x) + m_2^2(x) \big)^{ \alpha/2}}\, \gamma(dx) \\[2mm]
			&=\frac{1}{2} \int_{\RR}m_1^{\alpha/2}(x)\, m_2^{\alpha/2}(x)\sgn(m_1(x)m_2(x))\, \dd x.
			\end{align*}
			Thus,
			\begin{align*}
			\rho_t =  \int_{\mathbb{S}^1}  |s_1 s_2|^{\alpha/2} \sgn(s_1 s_2)\ \Gamma(d(s_1, s_2)) =  \int_{\RR} m_1^{\alpha/2}(x)\, m_2^{\alpha/2}(x)\sgn(m_1(x)m_2(x))\, \dd x.
			\end{align*}
		\end{pf}
	\end{Lemma}
	
	Now, a sufficient condition for the short range dependence of $X$ can be formulated in terms of $\rho_t$ or, equivalently, in terms of the kernel function $m$.
	
	\begin{Theorem}[]
		\label{thm: SRD_rho}
		Let $X= \{ X(t), t \in \mathbb{R} \}$ be a S$\alpha$S moving average process with parameter $\alpha \in (0, 2)$, nonnegative kernel function $m$ and $\alpha$-spectral covariance $\rho_t$ given in \eqref{eq: rho_t}. $X$ is SRD if
		\begin{align}
		\rho_t  \in \lp{1}{\mathbb{R}}, \label{eq: sufficient conditions rho} 
		\end{align}
		or, equivalently, $m \in \lp{\alpha/2}{\mathbb{R}}$.

		\begin{pf}
		Without loss of generality, assume $\mu$ is a probability measure. Now, apply \Cref{thm: SRD and LRD} to $X$ for some $\varepsilon \in (0,\ANA{m})$ and choose $A = \{ t \in \mathbb{R} \ \vert \ \rho_t \in [0, \varepsilon) \}.$ 	It follows from the integrability of $\rho_t$ that $\rho_t \rightarrow 0$ as $t \rightarrow \pm \infty$. Thus, there exists a constant $\tilde{t} >0$ such that $\rho_t < \varepsilon$ for all $t \in \mathbb{R}$ where $\vert t \vert > \tilde{t}$. Hence, it holds that $A^c \subset \big\{ t \in \mathbb{R}\ \big\vert\ \vert t \vert \leq \tilde{t} \big\}$ and $|A^c| < \infty.$
			
			Obviously, the right-hand side of the equality in \eqref{eq: SRD_or_LRD2} is bounded by 
			\begin{align}
			\nonumber&\frac{1}{2 \pi^2} \int_{A} \int_{\mathbb{R}_+} \int_{\mathbb{R}_+}   \frac{|\CFB{+}{-} - \CFB{+}{+}|}{s_1s_2}  \underbrace{\big\vert \re{1} \re{2} \big\vert}_{\leq 1} \dd s_1 \, \dd s_2 \dd t\\[2mm]
			\nonumber + &\frac{1}{2 \pi^2} \int_{A} \int_{\mathbb{R}_+} \int_{\mathbb{R}_+}  \frac{|\CFB{+}{-} + \CFB{+}{+} - 2\CFA{s_1}\CFA{s_2}|}{s_1s_2}  \underbrace{\big\vert \im{1} \im{2} \big\vert}_{\leq 1}  \, \dd s_1 \, \dd s_2\, \dd t\\
			\nonumber \leq &\frac{1}{ \pi^2} \int_{A} \int_{\mathbb{R}_+} \int_{\mathbb{R}_+}   \frac{\vert\CFB{+}{-} - \CFA{s_1}\CFA{s_2}\vert}{s_1s_2}   \dd s_1 \, \dd s_2 \dd t\\[2mm]
			+ &\frac{1}{ \pi^2} \int_{A} \int_{\mathbb{R}_+} \int_{\mathbb{R}_+}  \frac{\vert\CFB{+}{+} - \CFA{s_1}\CFA{s_2}\vert}{s_1s_2}    \, \dd s_1 \, \dd s_2\, \dd t=:\frac{1}{ \pi^2}\left( I_1+I_2\right).
			\end{align}
			
			
		    By inequalities \eqref{eq: I_1 finite} and \eqref{eq: I_2 finite} in Lemma \ref{lemma: ineq}  we get
			\begin{equation*}
			I_1,I_2 \leq \frac{8\pi}{\alpha^2 \Vert m \Vert_{\alpha}^{2\alpha}} \int_{A} \frac{\rho_t}{\sqrt{\Vert m \Vert_{\alpha}^{2\alpha}-\rho_t^2}} \, \dd t \leq \frac{8\pi}{\alpha^2 \Vert m \Vert _{\alpha}^{2\alpha}} \frac{1}{\sqrt{\Vert m \Vert_{\alpha}^{2\alpha}-\varepsilon^2}} \int_{A} \rho_t \, \dd t < \infty. 
			\end{equation*}
 	        Next, show that  condition \eqref{eq: sufficient conditions rho} holds true iff $m\in\lp{\alpha/2}{\mathbb{R}}$. By Fubini's theorem  we get
			\begin{align*}
			\int_{\mathbb{R}} \rho_t\, \dd t &=  \int_{\mathbb{R}} \int_{\mathbb{R}} m^{\alpha/2}(-x)m^{\alpha/2}(t-x)\, \dd x \, \dd t\\[2mm]
			&= \int_{\mathbb{R}}  m^{\alpha/2}(-x) \bigg( \int_{\mathbb{R}} m^{\alpha/2}(t-x) \, \dd t \bigg) \, \dd x\\[2mm]
			&= \bigg( \int_{\mathbb{R}}  m^{\alpha/2}(-x) \, \dd x \bigg)^2 = \Vert m \Vert _{\alpha/2}^{\alpha}<\infty	.
			\end{align*} 

\end{pf} 
\end{Theorem}

	Naturally, one may also ask for sufficient conditions for the long range dependence of $X$. Such a condition is given by 
	\begin{Theorem}[]
		\label{thm: LRD_min_m}
		Let $X= \{ X(t), t \in \mathbb{R} \}$ be a S$\alpha$S moving average process  with parameter $\alpha \in (0, 2)$ and nonnegative kernel function $m$. Then, $X$ is long range dependent if
		\begin{align}
		\int_{\mathbb{R}} \int_{\mathbb{R}}  (m^\alpha(x) \wedge m^\alpha(t)) \, \dd x \, \dd t = \infty. \label{eq: LRD_min_m}
		\end{align}
		\begin{pf}
			Given in the appendix.
		\end{pf}
	\end{Theorem}
    Additionally, if the kernel function $m$ is eventually monotonic, then we can simplify condition \eqref{eq: LRD_min_m} as follows.
	
	\begin{Corollary}[]
		\label{Corr: LRD}
		Let $X= \{ X(t), t \in \mathbb{R} \}$ be a S$\alpha$S moving average process with parameter $\alpha \in (0, 2)$ and nonnegative kernel function $m \in \lp{\alpha}{\mathbb{R}}$ which is eventually monotonic, i.e.   there is a number $a > 0 $ such that $m$ is monotonically decreasing on $(a, \infty)$ or monotonically increasing on $(-\infty, -a)$. Then, $X$ is long range dependent if
		\begin{align}
		\int_{a}^{\infty} t \,m^{\alpha}(t) \, \dd t = \infty \quad \text{ or } \quad \int_{-\infty}^{-a} t \,m^{\alpha}(t) \, \dd t = -\infty. \label{eq: sufficient_monotonicity}
		\end{align}
		Additionally, if $m$ is symmetric, the two sufficient conditions \eqref{eq: sufficient_monotonicity} are equivalent.
	\end{Corollary}	

		\begin{proof}
		Suppose $m$ is monotonically decreasing on $(a,\infty)$ and compute the integral \eqref{eq: LRD_min_m}. Thus, we have
		\begin{align*}
		&\int_{\mathbb{R}} \int_{\mathbb{R}} (m^\alpha(x) \wedge m^\alpha(t)) \, \dd x \, \dd t \geq  \int_{a}^{\infty} \bigg( \int_{a}^{t} m^{\alpha}(t) \, \dd x + \int_{t}^{\infty} m^{\alpha}(x) \, \dd x \bigg) \, \dd t \geq \int_{a}^{\infty} t m^{\alpha}(t)\, \dd t - a \Vert m \Vert ^{\alpha}_{\alpha}.
		\end{align*}
		The claim follows from the fact that $m \in \lp{\alpha}{\mathbb{R}}$. The case of $m$ monotonically increasing on some interval $(-\infty, -a)$ follows analogously.
	
	\end{proof}
	Now let us give an example of a kernel function $m \in \lp{\alpha}{\mathbb{R} }$ whose corresponding S$\alpha$S moving average is long range dependent if $m \notin \lp{\alpha/2}{\mathbb{R}}$.
	
	\begin{Example}[]
		\label{ex: LRD}
		Suppose we have a S$\alpha$S moving average process $X=\{  X(t), \, t \in \mathbb{R} \}$ with parameter $\alpha \in (0, 2)$ and nonnegative kernel function $m(x) \sim C \vert x \vert  ^{-\delta}$ as $\vert x \vert \rightarrow \infty$
		where $\delta > \frac{1}{\alpha}$ and $C > 0$.
		Obviously, $m \in \lp{\alpha}{\mathbb{R}}$ and $ m(x) \geq \frac{C}{2} \vert x \vert ^{-\delta}   $ where $\vert x \vert \geq a$ for some $a > 0$. Notice that
		\begin{align*}
		\int_{a}^{\infty} t \cdot \bigg( \frac{C}{2} \vert t \vert ^{-\delta} \bigg)^{\alpha} \, \dd t = \bigg( \frac{C}{2} \bigg)^{\alpha}\int_{a}^{\infty}  t^{1-\delta\alpha} \, \dd t = \infty
		\end{align*}
		if $1-\delta\alpha \geq -1$ or, equivalently, $\delta \leq \frac{2}{\alpha}$. Analogously to the proof of \Cref{Corr: LRD}, this implies that $m$ fulfills equation \eqref{eq: LRD_min_m} if $m \notin \lp{\alpha/2}{\mathbb{R}}$. 
		Thus, $X$ is long range dependent if $\delta \in\left(\frac{1}{\alpha}, \frac{2}{\alpha}\right]$ by \Cref{thm: LRD_min_m} and is short range dependent if $\delta > \frac{2}{\alpha}$ by Theorem \ref{thm: SRD_rho}.
	\end{Example}
	
	\begin{Remark}[]
	\label{rem: iff condition}
		On one hand, this example was given to illustrate that the conditions \eqref{eq: sufficient_monotonicity} in \Cref{Corr: LRD} are useful when the kernel function itself is not eventually monotonic but rather asymptotically equivalent to such a function. On the other hand, this example motivates our conjecture that a S$\alpha$S moving average is long range dependent iff $m \notin \lp{\alpha/2}{\mathbb{R}}$. However, the conjecture's proof is still to be found.
	\end{Remark}
	
	
	Similar results as above can be obtained for symmetric $\alpha$-stable linear time series $Y$.
	\begin{Definition}[\cite{Char_Fct_Hosoya}]
		\label{def: linear process}
		Let $\{Z(t), t\in \mathds{Z}\}$ be a sequence of i.i.d.~S$\alpha$S random variables with characteristic function $\CFA{s}=\exp\{-\tau^{\alpha}|s|^{\alpha}\}$, $\tau>0$, $0<\alpha < 2$,
		$s\in \mathbb{R}$. Let $\{a_{j}, j\in \mathds{Z}\}$ be a sequence of numbers such that $\sum_{j=-\infty}^{+\infty}|a_j|<\infty$. The stochastic process defined by
		\begin{align}
		Y(t)=\sum\limits_{j=-\infty}^{+\infty}a_jZ(t-j),\quad  t \in \mathds{Z}, \label{eq: linear process}
		\end{align}
		is called a {\it  linear S$\alpha$S time series}.
	\end{Definition}         
	
	Notice that  $Y$ can be written as a continuous time S$\alpha$S moving average with parameter $\alpha \in (0, 2)$ and kernel function $m(x) = \sum_{j=-\infty}^{\infty} a_j \mathds{1}_{[(j-1)\tau^{\alpha},j\tau^{\alpha})}(x)$ sampled at time instants $t \in \mathbb{Z}$. Thus, $Y$ is positively associated if the coefficients $a_j$ are nonnegative for all $j \in \mathbb{Z}$. Moreover, the function $\rho_t$  simplifies  to $\rho_t= \sum_{j=-\infty}^{\infty}  a^{\alpha/2}_{-j} a^{\alpha/2}_{t-j}$, $t\in \mathds{Z}$.
	
	\begin{Remark}
		\label{rem: time series}
		Theorems \ref{thm: SRD_rho} and \ref{thm: LRD_min_m} as well as \Cref{Corr: LRD} apply to linear processes with the obvious substitute of $\sum_{t=-\infty}^{\infty}$ for $\int_{\mathbb{R}}\, \dd t$. Indeed, 
		let $Y= \{ Y(t), t \in \mathbb{Z} \}$ be  a stationary S$\alpha$S time series  with parameter $\alpha < 2$ and nonnegative coefficients $\{ a_j, j \in \mathbb{Z} \}$. If
		\begin{align}
		\sum_{t = - \infty}^{\infty} \rho_t  < \infty \quad \text{ or, equivalently, } \quad \sum_{j=-\infty}^{\infty} a_j ^{\alpha/2} < \infty,  \label{eq: sufficient conditions SRD linear} 
		\end{align}
		then $Y$ is short range dependent. If 
		\begin{align}
		\sum_{j=-\infty}^{\infty} \sum_{t=-\infty}^{\infty}   (a_j^\alpha \wedge a_t^\alpha) = \infty \label{eq: LRD_min_a},
		\end{align}
		then $Y$ is long range dependent. Additionally, if the coefficients  $a_j$ are  for some $a > 0$ monotonically increasing for all $j < -a$ or monotonically decreasing for all $j > a$ , then $Y$ is long range dependent if
		$
		\sum_{j=a}^{\infty} j\,a_j^{\alpha} = \infty$ or $\sum_{j=-\infty}^{a} j\,a_j^{\alpha} = -\infty $, respectively.
	\end{Remark}
	
	
	\section{Long Range Dependence of Max-stable Processes}
	\label{sec: MaxStable}
	
	In this section, we demonstrate that it is possible to use already existing dependence properties to check \Cref{def: LRD_Spodarev} instead of inverting characteristic functions as in the previous sections.

	Any max-stable process is positively associated, see for instance Proposition 5.29 in \cite{resnick-87}. Its dependence properties are typically summarized by its pairwise extremal coefficients $\{\theta_t, \ t \in T\}$ defined via
	$$ \PP(X(0) \leq x, \ X(t) \leq x) = \PP(X(0) \leq x)^{\theta_t} \quad \text{for all } x > 0,$$
	cf.\ \cite{schlather-tawn-03}. By a series expansion of the logarithm, it can be
	seen that $ \theta_t = 2 - \lim_{x \to \infty} \PP(X(t) > x \mid X(0) > x)$.
	Thus, $\theta_t \in [1,2]$, where $\theta_t=2$ corresponds to the case
	of (asymptotic) independence between $X(0)$ and $X(t)$ while $\theta_t=1$ means
	full dependence. Even though the joint distribution of $(X(0),X(t))$ is not
	uniquely determined by $\theta_t$, this characteristic turns out to be a useful
	tool for the identification of dependence properties. For instance, \cite{stoev-08}, \cite{kabluchko-schlather-10} and \cite{dombry-kabluchko-17}  provide necessary and sufficient conditions for ergodicity and mixing of a max-stable process in terms of its pairwise extremal coefficients.
	
	Here, we focus on the property of long-range dependence given by \Cref{def: LRD_Spodarev}. We obtain bounds for $\Cov \big(\mathds{1} \{ X(0) > u \}, \mathds{1} \{ X(t) > v \}\big)$, $t \in T$, $u,v > 0$, by making use of the following lemma.
	
	\begin{Lemma}[] \label{lem:bounds-cdf}
		Let $(X,Y)$ be a bivariate max-stable random vector with $\alpha$-Fr\'echet
		margins, $\alpha > 0$, and extremal coefficient $\theta \in [1,2]$. Then, we have
		\begin{align*}
		\exp\left( - \frac 1 {u^\alpha} - \frac 1 {v^\alpha} + \frac{2 -\theta}{(u \vee v)^\alpha} \right)
		\leq{} \PP(X \leq u, Y \leq v) 
		\leq{} \exp\left( - \frac 1 {u^\alpha} - \frac 1 {v^\alpha} + \frac{2 -\theta}{(u \wedge v)^\alpha} \right)
		\end{align*} 	
		for all $u,v > 0$.
	\end{Lemma}	
	
	\begin{proof}
		It is well-known that the cumulative distribution function of a bivariate 
		max-stable random vector $(X,Y)$ with $\alpha$-Fr\'echet margins is of the 
		form
		$$ \PP(X \leq u, Y \leq v) = \exp\left(- \EE\left[ \left( \frac{W_X}{u} \vee \frac{W_Y}{v} \right)^\alpha \right] \right), \quad u,v \geq 0, $$
		where $(W_X,W_Y)$ is a bivariate random vector with $\EE W_X^\alpha = \EE W_Y^\alpha = 1$ 
		cf.~Chapter 5 in \cite{resnick-87}, for instance. This so-called \textit{spectral
			vector} $(W_X,W_Y)$ is closely connected to the extremal coefficient via the 
		relation $\theta = \EE( W_X^\alpha \vee W_Y^\alpha )$. Thus, we obtain
		\begin{align}
		& -\log \PP(X \leq u, Y \leq v) {}={} \EE\left[ \left(\frac{W_X} u\right)^\alpha \vee \left(\frac{W_Y} v\right)^\alpha \right] \nonumber\\
		={}& \EE \bigg[ \phantom{\vee} \left( \frac{W_X^\alpha}{(u \vee v)^\alpha} + W_X^\alpha \cdot \left( \frac 1 {(u \wedge v)^\alpha} - \frac 1 {(u \vee v)^\alpha} \right) \cdot \mathbf{1}\{u \leq v\} \right) \nonumber \\  
		& \phantom{\EE \bigg[} \vee \left( \frac{W_Y^\alpha}{(u \vee v)^\alpha} + W_Y^\alpha \cdot \left( \frac 1 {(u \wedge v)^\alpha} - \frac 1 {(u \vee v)^\alpha} \right) \cdot \mathbf{1}\{v < u\} \right) \bigg]. \label{eq:lem}
		\end{align}     
		Distinguishing between the two cases  $u \leq v$ and $v < u$, it can be seen
		that the right-hand side of equation \eqref{eq:lem} can be bounded from above by 
		\begin{align*}    
		& \frac{1}{(u \vee v)^\alpha} \EE\left( W_X^\alpha \vee W_Y^\alpha \right) 
		+ \left( \frac 1 {(u \wedge v)^\alpha} - \frac 1 {(u \vee v)^\alpha} \right) \cdot \EE\left( W_X^\alpha \cdot \mathbf{1}\{u \leq v\} + W_Y^\alpha \cdot \mathbf{1}\{v < u\}\right) \\
		={}& \frac{\theta}{(u \vee v)^\alpha} + \left( \frac 1 {u^\alpha} + \frac 1 {v^\alpha} - \frac 2 {(u \vee v)^\alpha} \right) \cdot 1 
		{}={} - \frac{2-\theta}{(u \vee v)^\alpha} + \frac{1}{u^\alpha} + \frac{1}{v^\alpha},
		\end{align*}
		where we used the fact that $\EE W_X^\alpha = \EE W_Y^\alpha = 1$. 
		This gives the lower bound stated in the lemma. Analogously, we obtain
		\begin{align*}
		& -\log \PP(X \leq u, Y \leq v) {}={} \EE\left[ \left(\frac{W_X} u\right)^\alpha \vee \left(\frac{W_Y} v\right)^\alpha \right] \\
		={}& \EE \bigg[ \phantom{\vee} \left( \frac{W_X^\alpha}{(u \wedge v)^\alpha} - W_X^\alpha \cdot \left( \frac 1 {(u \wedge v)^\alpha} - \frac 1 {(u \vee v)^\alpha} \right) \cdot \mathbf{1}\{v < u\} \right) \\  
		& \phantom{\EE \bigg[} \vee \left( \frac{W_Y^\alpha}{(u \wedge v)^\alpha} - W_Y^\alpha \cdot \left( \frac 1 {(u \wedge v)^\alpha} - \frac 1 {(u \vee v)^\alpha} \right) \cdot \mathbf{1}\{u \leq v\} \right) \bigg] \\   
		\geq{}& \frac{1}{(u \wedge v)^\alpha} \EE\left( W_X^\alpha \vee W_Y^\alpha \right) 
		- \left( \frac 1 {(u \wedge v)^\alpha} - \frac 1 {(u \vee v)^\alpha} \right) \cdot \EE\left( W_X^\alpha \cdot \mathbf{1}\{v < u\} + W_Y^\alpha \cdot \mathbf{1}\{u \leq v\}\right) \\
		={}& \frac{\theta}{(u \wedge v)^\alpha} - \left( \frac{2}{(u \wedge v)^\alpha} - \frac 1 {u^\alpha} - \frac 1 {v^\alpha} \right) \cdot 1
		= - \frac{2-\theta}{(u \wedge v)^\alpha} + \frac 1 {u^\alpha} + \frac 1 {v^\alpha},
		\end{align*}
		which gives the corresponding upper bound.
	\end{proof}
	
	\begin{Remark}[]
		Note that the lower bound in Lemma \ref{lem:bounds-cdf} corresponds to the bound 
		given in \cite{strokorb-schlather-15}. For the so-called Molchanov--Tawn model,
		we even have
		\begin{align*}
		\PP(X \leq u, Y \leq v) ={}
		\exp\left( - \frac 1 {u^\alpha} - \frac 1 {v^\alpha} + \frac{2 -\theta}{(u \vee v)^\alpha} \right), \quad u, v > 0,
		\end{align*} 
		i.e.\ the bound is sharp in this case.
	\end{Remark}
	
	Using \Cref{lem: Covariance with joint and marginal distr}, Lemma \ref{lem:bounds-cdf} yields the following bounds for $\Cov \big(\mathds{1} \{ X(0) > u \}, \mathds{1} \{ X(t) > v \}\big)$:
	\begin{align*}
	\exp\left(- \frac 1 {u^\alpha} - \frac 1 {v^\alpha}\right) \cdot
	\left[ \exp\left(\frac{2 -\theta_t}{(u \vee v)^\alpha} \right) - 1\right] 
	\leq{}& \Cov \big(\mathds{1} \{ X(0) > u \}, \mathds{1} \{ X(t) > v \}\big) \\
	\leq{}& \exp\left(- \frac 1 {u^\alpha} - \frac 1 {v^\alpha}\right) \cdot
	\left[ \exp\left(\frac{2 -\theta_t}{(u \wedge v)^\alpha} \right) - 1\right] 
	\end{align*}
	for $u, v > 0$ and $\Cov \big(\mathds{1} \{ X(0) > u \}, \mathds{1} \{ X(t) > v \}\big) = 0$ if $u \wedge v \leq 0$ due to the $\alpha$-Fr\'echet margins.
	Consequently, one obtains for the integral in \eqref{eq: LRD_Spodarev} that
	\begin{align}
	& \int_{\mathbb{R}_+} \int_{\mathbb{R}_+} \exp\left(- \frac 1 {u^\alpha} - \frac 1 {v^\alpha}\right) 
	\int_{T}  \left[\exp\left(\frac{2 -\theta_t}{(u \vee v)^\alpha} \right) - 1 \right] \, \dd t 
	\, \mu(\dd u) \, \mu(\dd v) \nonumber \\
	\leq{}& \int_{T} \int_{\RR_+} \int_{\RR_+} \big\vert\Cov \big(\mathds{1} \{ X(0) > u \}, \mathds{1} \{ X(t) > v \}\big)\big\vert \, \mu(\dd u) \, \mu(\dd v) \, \dd t \nonumber \\
	\leq{}& \int_{\mathbb{R}_+} \int_{\mathbb{R}_+} \exp\left(- \frac 1 {u^\alpha} - \frac 1 {v^\alpha}\right) 
	\int_{T}  \left[\exp\left(\frac{2 -\theta_t}{(u \wedge v)^\alpha} \right) - 1 \right] \, \dd t 
	\, \mu(\dd u) \, \mu(\dd v). \label{eq:bounds-int-cov}
	\end{align}

	From the lower and the upper bound in \eqref{eq:bounds-int-cov}, we directly obtain a necessary and sufficient condition for long range dependence. Interestingly, unlike in case of $\alpha$-stable processes, the criterion does not depend on $\alpha>0$.
	
	\begin{Theorem}[] \label{prop:lrd}
		Let $X=\{X(t), \ t \in T\}$ be a stationary max-stable process with $\alpha$-Fr\'echet marginal distributions, $\alpha > 0$, 
		and pairwise extremal coefficients $\{\theta_t, \, t \in T\}$.
		Then, $X$ is long range dependent if and only if 
		\begin{equation} \label{eq:cond-lrd}
		\int_{T} (2-\theta_t) \, \dd t = \infty. 
		\end{equation}.
	\end{Theorem}
	\begin{proof}
		First, assume that \eqref{eq:cond-lrd} holds. Choosing the finite measure $\mu = \delta_{\{1\}}$ as the Dirac measure, we obtain from the lower bound in \eqref{eq:bounds-int-cov} and the inequality $ \exp(x) \geq 1 + x$ for all $x \geq 0$ that
		\begin{align*}
		& \int_{T} \int_{\RR} \int_{\RR} |\Cov \big(\mathds{1} \{ X(0) > u \}, \mathds{1} \{ X(t) > v \}\big)| \, \delta_{\{1\}}(\dd u) \, \delta_{\{1\}}(\dd v) \, \dd t \\
		\geq{}& \int_{\mathbb{R}_+} \int_{\mathbb{R}_+} \exp\left(- \frac 1 {u^\alpha} - \frac 1 {v^\alpha}\right) 
		\int_{T}  \left[\exp\left(\frac{2 -\theta_t}{(u \vee v)^\alpha} \right) - 1 \right] \, \dd t 
		\, \delta_{\{1\}}(\dd u) \, \delta_{\{1\}}(\dd v) \\
		={}& \exp(-2) \cdot  \int_{T}  \left[\exp\left(2 -\theta_t\right) - 1 \right] \, \dd t \geq  \exp(-2) \cdot  \int_{T}  (2 -\theta_t) \, \dd t = \infty.
		\end{align*}
		
		Conversely, assume that \eqref{eq:cond-lrd} does not hold, i.e.\
		$$ C = \int_{T} (2 -\theta_t) \, \dd t < \infty. $$
		As  $0 \leq 2 - \theta_t \leq 1$ for all $t \in T$, we obtain that
		\begin{align*}
		\int_T (2-\theta_t)^k \, \dd t \leq \int_T (2-\theta_t) \, \dd t = C
		\end{align*}
		for all $k \in \NN$ and therefore
		\begin{align*}
		& \exp(-u^{-\alpha}) \int_T \left[ \exp((2-\theta_t)u^{-\alpha}) -1\right] \, \dd t {}={} \exp(-u^{-\alpha}) \int_T \sum_{k=1}^\infty \frac {u^{-\alpha k}} {k!} (2 -\theta_t)^k \, \dd t \\
		={}& \exp(-u^{-\alpha}) \sum_{k=1}^\infty \frac {u^{-\alpha k}} {k!} 
		\int_T (2-\theta_t)^k \, \dd t {}\leq{} \exp(-u^{-\alpha})  \sum_{k=1}^\infty \frac {u^{-\alpha k}} {k!} C \leq C
		\end{align*}
		for all $u \geq 0$. Combining this inequality with the upper bound in \eqref{eq:bounds-int-cov}, we have
		\begin{align*}
		& \int_{T} \int_{\RR} \int_{\RR} \big\vert\Cov \big(\mathds{1} \{ X(0) > u \}, \mathds{1} \{ X(t) > v \}\big)\big\vert \, \mu(\dd u) \, \mu(\dd v) \, \dd t \\
		\leq{}& \int_{\mathbb{R}_+} \int_{\mathbb{R}_+} \int_{T}\exp\left(- \frac 1 {u^\alpha} - \frac 1 {v^\alpha}\right) 
		\int_{T}  \left[\exp\left(\frac{2 -\theta_t}{(u \wedge v)^\alpha} \right) - 1 \right] \, \dd t \, \mu(\dd u) \, \mu(\dd v) \\
		\leq{}& \int_{\mathbb{R}_+} \int_{\mathbb{R}_+} \int_{T}  \exp\left(- \frac 1 {(u \wedge v)^\alpha} \right) 
		\int_{T}  \left[\exp\left(\frac{2 -\theta_t}{(u \wedge v)^\alpha} \right) - 1 \right] \, \dd t \, \mu(\dd u) \, \mu(\dd v) \\
		\leq{}& \int_{\mathbb{R}_+} \int_{\mathbb{R}_+} C \, \mu(\dd u) \, \mu(\dd v) \leq C  \mu^2(\mathbb{R}_+) 
		\end{align*}
		for any finite measure $\mu$ on $\RR$. Thus, $X$ is short range dependent.
	\end{proof}
	
	\begin{Example}
		Here, we consider two popular examples of max-stable processes, namely the extremal Gaussian process and the Brown--Resnick process.
		\begin{enumerate}
			\item The extremal Gaussian process \citep{schlather02} is a max-stable process with
				$1$-Fr\'echet marginal distributions and finite-dimensional distributions of the form
				$$ \PP(X(t_1) \leq x_1, \ldots, X(t_d) \leq x_d)
				= \exp\left( - \sqrt{2\pi} \max_{i=1,\ldots,d} \frac{\max\{W(t_i),0\}}{x_i} \right),
				\quad t_i \in T, \, x_i > 0, $$
				where $\{W(t), \, t \in T\}$ is a centered stationary Gaussian process on $T=\RR$. The 
				extremal coefficients of the extremal Gaussian process are given by
			$$ \theta_t = 1 + \sqrt{1 - \frac{1+\rho_t}{2}}, \quad t \in T,$$
			where $\rho_t = \Corr(W(t),W(0))$ denotes the correlation function of the underlying Gaussian process $W$. Provided that $\rho_t \geq 0$ for all $t \in T$, we have that $\theta_t \leq 1 + \sqrt{1/2}$, and, consequently,
			$$ \int_{T} (2 -\theta_t) \, \dd t \geq \int_{\RR} (1 - \sqrt{1/2}) \, \dd t = \infty,$$
			that is, the process is long range dependent by \Cref{prop:lrd}.
			\item The Brown--Resnick process \citep{KSH09} is a max-stable process with $1$-Fr\'echet marginal distributions and finite-dimensional distributions of
				the form 
				$$ \PP(X(t_1) \leq x_1, \ldots, X(t_d) \leq x_d)
				= \exp\left( - \max_{i=1,\ldots,d} \frac{\exp(W(t_i) - \frac 1 2 \Var[W(t_i)])}{x_i} \right), $$
				$t_i \in T, \, x_i > 0,$ where $W$ is a centered Gaussian process with stationary increments on $T=\RR$. The extremal coefficients of the Brown--Resnick process can be expressed in terms of the 
			variogram $\gamma(t) = \EE[(W(t)-W(0))^2]$, $t \in \RR$, of the underlying Gaussian process $W$ via the relation
			$$ \theta_t = 2 \Phi\left( \frac{\sqrt{\gamma(t)}}{2} \right), \quad t \in \RR,$$
			where $\Phi$ denotes the standard normal distribution function. 
			
			Now assume that there exists some constant $C > 8$ such that
			$\gamma(t) \geq C \log |t|$ for $|t|$ being sufficiently large. From Mill's ratio $1 - \Phi(x) \sim x^{-1} \varphi(x)$ as $x \to \infty$
			with $\varphi$ being the standard normal density function, it follows that
			\begin{align*}
			2 - \theta_t {}={}& 2 [ 1 - \Phi(\sqrt{\gamma(t)}/2)]
			{}\leq{} 2 [1 - \Phi(\sqrt{C \log|t|} / 2)] \\
			\sim{}& \frac{4}{\sqrt{C \log|t|}} \varphi(\sqrt{C \log|t|} / 2) 
			{}={} \frac{2 \sqrt{2}}{\sqrt{\pi C \log|t|}} |t|^{-C/8}, \quad |t| \to \infty,
			\end{align*}
			is integrable. Thus, by \Cref{prop:lrd}, the Brown--Resnick process is SRD if $$\liminf_{|t| \to \infty} \gamma(t) / \log |t| > 8,$$ which is true, for instance, for any fractional Brownian motion $W$.
			
			If, in contrast, the variogram $\gamma$ of the underlying Gaussian process $W$ is bounded as in the case of a stationary process, we obtain that $\sup_{t \in T} \theta_t < 2$. Thus, analogously to the case of the extremal Gaussian process, the Brown-Resnick process can be shown to be LRD.
			
			Note that these conditions also appear in the literature when analyzing the existence of a mixed moving maxima (M3) representation of a Brown-Resnick process:
			In \cite{KSH09}, it is shown that a M3 representation exists if $\liminf_{|t| \to \infty} \gamma(t) / \log |t| > 8$.
			In case of a bounded variogram, however, the resulting Brown-Resnick is not even mixing. As sufficient and necessary conditions for the existence of a M3 representation
			are stated in terms of the asymptotic behavior of the sample paths of the underlying Gaussian process rather than in terms of its variogram \citep[cf.][for instance]{wang-stoev-10, dombry-kabluchko-17},
			to the best of our knowledge, there is no general treatment of the gap between these two cases. Similarly, for SRD/LRD,  a detailed analysis of further cases is beyond the scope of this paper. 
		\end{enumerate}	
		
	\end{Example}

	\begin{Remark}
		In this section, using known dependency properties allows to avoid complex calculation such that no restrictions on the index set $T$ are required. In particular, all the results are also valid for max-stable random fields, i.e.\ the case that $T \subset \RR^d$.
	\end{Remark}

	\section{Application to Data}
	\label{sec: data}
	In this section, we want to motivate our theoretical results by showing their applicability to real world data. To do so, let us consider the daily log-returns of the Intel corporation share from Mar 03, 2013 to Aug 21, 2017 depicted in \Cref{fig: log-returns}.
	\begin{figure}
		\centering
		\includegraphics[width=0.8\linewidth]{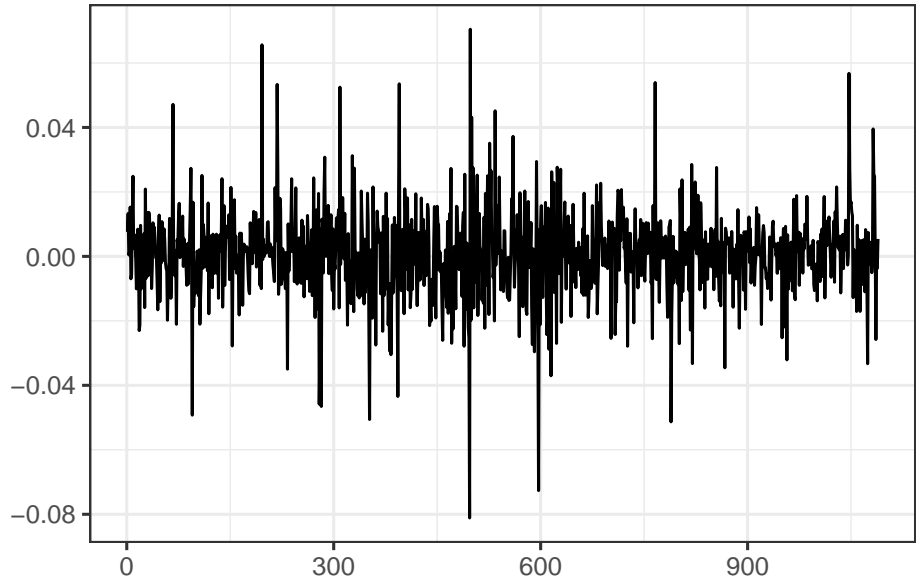}
		\caption{Daily log-returns based on the opening price of the Intel corporation share from Mar 03, 2013 to Aug 21, 2017}
		\label{fig: log-returns}
	\end{figure}
	Preliminary analysis has shown that the marginal distribution of these log-returns fits reasonably well to that of a symmetric $\alpha$-stable distribution with estimated index of stability $\hat{\alpha} = 1{.}56$ and scale parameter $\hat{\sigma} =  0{.}0072$ as depicted in \Cref{fig: marginal-dist}. For simplicity, here we use the simple and consistent estimation procedure proposed by \cite{McCulloch86}.
	\begin{figure}
		\centering
		\includegraphics[width=0.8\linewidth]{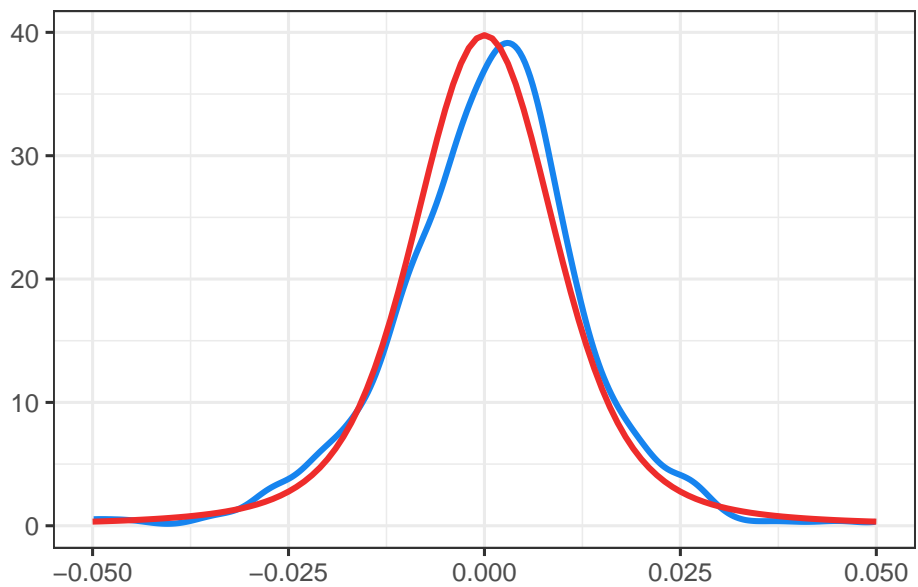}
		\caption{Estimated density of the log-returns  (in blue) compared to the theoretical density of a symmetric $\alpha$-stable distribution with index of stability $\hat{\alpha} = 1{.}56$ and scale parameter $\hat{\sigma} =  0{.}0072$ (in red).}
		\label{fig: marginal-dist}
	\end{figure}
	
	Further, we model this time series using a linear S$\alpha$S process $Y(t)=\sum_{j=-\infty}^{+\infty}a_jZ(t-j)$,  $t \in \mathds{Z}$ with $\alpha \in (0, 2)$, as described in \Cref{def: linear process}. By \Cref{rem: time series}, we can apply our previous continuous-time results from \Cref{sec: AlphaStable} by considering a continuous-time S$\alpha$S moving average $X$ with a piecewise constant kernel function and interpreting the time series $Y$ as $X$ sampled at time instances $t \in \mathbb{Z}$.

	To do so, we estimate a continuous-time kernel function by a non-parametric approach and check the conditions in Theorems \ref{thm: SRD_rho} and \ref{thm: LRD_min_m}. By \Cref{ex: LRD}, it suffices to check the  condition  $m \in \lp{\alpha/2}{\mathbb{R}}$ , if the estimated kernel function exhibits power decaying tails.   \\
	However, to the best of our knowledge, there is no universally applicable non-parametric approach for kernel estimation in our setting. For instance, the procedure proposed by \cite{Kampf2020} estimates the kernel of a S$\alpha$S moving average under certain conditions posed on the underlying kernel function $m$. However, the authors of this particular paper conclude that under their assumptions, $m$ must be bounded and $m \in \lp{p}{\mathbb{R}}$ for all $p \in (1/a, \infty]$ where $a > \max \{ 2, 1/\alpha \}$ which, in particular, implies that $m \in \lp{\alpha/2}{\mathbb{R}}$. Consequently, \Cref{thm: SRD_rho} implies that this kernel estimation procedure is applicable in our setting only if $X$ is SRD.

	Therefore, let us choose a simple parametric  minimal contrast method based on the codifference \begin{align*}
	\tau(t) = \Vert X(0) \Vert ^{\alpha}_{c,\alpha} 
	+ 
	\Vert X(t) \Vert ^{\alpha}_{c,\alpha} 
	- 
	\Vert X(0) - X(t) \Vert ^{\alpha}_{c,\alpha}
	\end{align*} 
	of $X(0)$ and $X(t)$ as defined in \cite[Definition 2.10.1]{SamorodTaqqu94}. 
	Here $\Vert \cdot \Vert_{c,\alpha}$ describes the covariation norm of a S$\alpha$S random variable given by \cite[Definition 2.8.1]{SamorodTaqqu94}. 
	
	We assume that the kernel function of $X$ is causal, i.e., supported on the positive half-line, and  parametrized like
	\begin{align*}
		m(t) = \sum_{k=1}^{\infty} \frac{c}{1 + k^\delta} \mathds{1}\{ t \in [k-1, k), t \geq 0 \},
	\end{align*}
	where $c, \delta > 0$. By \Cref{ex: LRD}, the process $X$ is well-defined iff $\delta > \frac{1}{\alpha}$ and long range dependent iff $\delta \leq \frac{2}{\alpha}$.
	
	 By \cite[Example 3.6.2]{SamorodTaqqu94} we have that $\Vert X(0) \Vert ^{\alpha}_{c,\alpha} = \Vert X(t) \Vert^{\alpha}_{c,\alpha} = \Vert m \Vert ^{\alpha}_{\alpha}$ 
	 and 
	 $\Vert X(0) - X(t) \Vert ^{\alpha}_{c,\alpha} = \Vert m(\cdot) - m(t - \cdot) \Vert ^{\alpha}_{\alpha} $. 
	 By simple calculations, we get that for all $t \in \mathbb{N}$ 
	\begin{align*}
		\Vert m \Vert ^{\alpha}_{\alpha} 
		&= 
		c^{\alpha} \sum_{k=1}^{\infty} (1+k^{\delta})^{-\alpha},\\
		\Vert m(\cdot) - m(t - \cdot) \Vert ^{\alpha}_{\alpha} 
		&= 
		c^{\alpha} \bigg( 
			\sum_{k=1}^{t} (1 + k^{\delta})^{-\alpha} 
			+ 
			\sum_{k=1}^{\infty} \Big\vert (1 + k^{\delta})^{-1} 
			- 
			(1 + (k+t)^{\delta})^{-1} \Big\vert^{\alpha} 
		\bigg).
	\end{align*}
	Note here that despite the process $X$ being defined on the whole real line, we are only interested in sample time points $t \in \mathbb{N}$ for our data analysis. Consequently, the codifference $\tau(t)$ of $X(0)$ and $X(t)$ writes
	\begin{align}
		\tau(t) 
		=
		c^{\alpha} \bigg( 
			\sum_{k=1}^{\infty} (1+k^{\delta})^{-\alpha} 
			+  
			\sum_{k=t+1}^{\infty} (1 + k^{\delta})^{-\alpha} 
			- 
			\sum_{k=1}^{\infty} \Big\vert (1 + k^{\delta})^{-1} - (1 + (k+t)^{\delta})^{-1} \Big\vert^{\alpha}  
		\bigg). \label{eq: tau_theo}
	\end{align}
	By \cite[Prop. 2.8.2]{SamorodTaqqu94} it holds that $\Vert X(t) \Vert ^{\alpha}_{c,\alpha} = \sigma$ 
	and 
	$ \Vert X(0) - X(t) \Vert ^{\alpha}_{c,\alpha} = \sigma_t$ 
	for all $t \in \mathbb{N}$ where $\sigma$ and $\sigma_t$ are the scale parameters of $X(0)$ and $X(0)-X(t)$, respectively. We compare the theoretical quantity \eqref{eq: tau_theo} to
	\begin{align*}
		\hat{\tau}(t) =   2 \hat{\sigma} - \hat{\sigma_t} 
	\end{align*}  
	where $\hat{\sigma}$ and $\hat{\sigma_t}$ are estimators of $\sigma$ and $\sigma_t$, respectively. Again, we use the approach proposed by \cite{McCulloch86} to estimate $\sigma$. When estimating $\hat{\sigma_t}$, we do the same based on computed observations $X(i) - X(i + t)$, $i =1, \dots, n-t$, where $n$ denotes the length of the original sample. 
	
	Now, let us estimate the parameters $\delta$ and $c$ by minimizing the $\mathcal{L}^2$-distance of $\tau$ and $\hat{\tau}$ based on the first 25 time instances. We choose this value based on the computation costs of a higher threshold and on the fact that this value has proven useful in simulation studies as part of the preliminary analysis. Using our procedure, we estimate $\hat{\delta} = 1{.}0474$ and $\hat{c} = 0{.}0019$. To validate our estimation, we ran a parametric bootstrap and simulated 1000 S$\hat{\alpha}$S time series with kernel parameters $\hat{\delta}$ and $\hat{c}$ to get a grasp on the variance of $\tau$ in our chosen model. The results of our parametric bootstrap are depicted in \Cref{fig: bootstrap}. It shows that despite some discrepancies with regard to the empirically computed $\hat{\tau}$, which is to be expected in real data, our model exhibits a reasonable fit for the data set. 
	
 	Notice that, due to the empirical estimation of the scale parameters $\sigma$ and $\sigma_t$, it happens that $\hat{\tau}(t)<0$ for some $t>0$, which is impossible for the theoretical codifference $\tau$. This, however, does not substantially affect the quality of the fit of $\hat{\tau}$ to $\tau$. By the same parametric bootstrap we find the standard deviations of $0{.}1802$ for $\delta$ and $0{.}0005$ for $c$. It holds that $\hat{\delta} = 1{.}0474 < 2/\alpha = 1{.}2821$ which implies by \Cref{ex: LRD} that the log-returns of Intel Corporation are long range dependent.
	\begin{figure}
		\centering
		\includegraphics[width=0.85\linewidth]{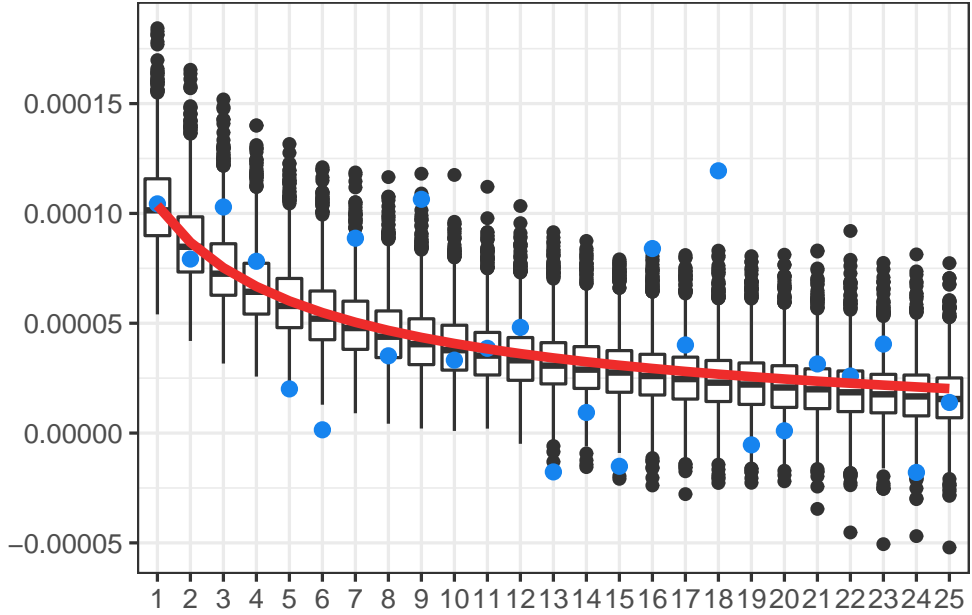}
		\caption{Parametric bootstrap of $\tau(t)$. Black: boxplots of $\hat{\tau}(t)$ based on 1000 S$\hat{\alpha}$S time series simulated with kernel parameters $\hat{\delta}$ and $\hat{c}$; Blue: $\hat{\tau}(t)$ based on data; Red: kernel function with parameters $\hat{\delta}$ and $\hat{c}$ (plotted as a line for better visibility).}
		\label{fig: bootstrap}
	\end{figure}

	\FloatBarrier
	\section*{Data Availability Statement}
	The data used in this paper was provided by a free data sample from \\ \href{https://www.quandl.com/data/EOD/INTC}{https://www.quandl.com/data/EOD/INTC}.
	\bibliographystyle{abbrvnat}
	\bibliography{lit,geostoch,abbrev} 

\newcommand{\noopsort}[1]{} \newcommand{\printfirst}[2]{#1}
  \newcommand{\singleletter}[1]{#1} \newcommand{\switchargs}[2]{#2#1}
\begin{thebibliography}{54}
\providecommand{\natexlab}[1]{#1}
\providecommand{\url}[1]{\texttt{#1}}
\expandafter\ifx\csname urlstyle\endcsname\relax
  \providecommand{\doi}[1]{doi: #1}\else
  \providecommand{\doi}{doi: \begingroup \urlstyle{rm}\Url}\fi

\bibitem[Asadi et~al.(2015)Asadi, Davison, and
  Engelke]{asadi-davison-engelke-18}
P.~Asadi, A.~C. Davison, and S.~Engelke.
\newblock Extremes on river networks.
\newblock \emph{Ann. Appl. Stat}, 9\penalty0 (4):\penalty0 2023--2050, 2015.

\bibitem[Avram and Taqqu(1986)]{Avram1986}
F.~Avram and M.~S. Taqqu.
\newblock Weak convergence of moving averages with infinite variance.
\newblock In \emph{Dependence in probability and statistics ({O}berwolfach,
  1985)}, volume~11 of \emph{Progr. Probab. Statist.}, pages 399--415.
  Birkh\"{a}user Boston, Boston, MA, 1986.

\bibitem[Beran et~al.(2012)Beran, Das, and Schell]{Beran2012}
J.~Beran, B.~Das, and D.~Schell.
\newblock On robust tail index estimation for linear long-memory processes.
\newblock \emph{J. Time Series Anal.}, 33\penalty0 (3):\penalty0 406--423,
  2012.

\bibitem[Beran et~al.(2013)Beran, Feng, Ghosh, and Kulik]{beran:kulik:2013}
J.~Beran, Y.~Feng, S.~Ghosh, and R.~Kulik.
\newblock \emph{Long memory processes. Probabilistic Properties and Statistical
  Methods}.
\newblock Springer, 2013.

\bibitem[Botcharova et~al.(2014)Botcharova, Farmer, and
  Berthouze]{botcharova2014markers}
M.~Botcharova, S.~F. Farmer, and L.~Berthouze.
\newblock Markers of criticality in phase synchronization.
\newblock \emph{Front. Syst. Neurosci.}, 8:\penalty0 176, 2014.

\bibitem[Buishand et~al.(2008)Buishand, De~Haan, Zhou,
  et~al.]{buishand-etal-08}
T.~Buishand, L.~De~Haan, C.~Zhou, et~al.
\newblock On spatial extremes: with application to a rainfall problem.
\newblock \emph{Ann. Appl. Stat}, 2\penalty0 (2):\penalty0 624--642, 2008.

\bibitem[Bulinski and Shashkin(2007)]{Bulinski_Shashkin}
A.~Bulinski and A.~Shashkin.
\newblock \emph{Limit Theorems for Associated Random Fields and Related
  Systems}.
\newblock World Scientific Publishing Co. Pte. Ltd., 2007.

\bibitem[Burnecki(2012)]{burnecki2012farima}
K.~Burnecki.
\newblock Farima processes with application to biophysical data.
\newblock \emph{J. Stat. Mech. Theory Exp}, 2012\penalty0 (05):\penalty0
  P05015, 2012.

\bibitem[Cheung and Lai(1995)]{cheung1995search}
Y.-W. Cheung and K.~S. Lai.
\newblock A search for long memory in international stock market returns.
\newblock \emph{J. Int. Money Finance}, 14\penalty0 (4):\penalty0 597--615,
  1995.

\bibitem[Coles(1993)]{coles93}
S.~G. Coles.
\newblock Regional modelling of extreme storms via max-stable processes.
\newblock \emph{J. R. Stat. Soc. Ser. B. Stat. Methodol.}, 55\penalty0
  (4):\penalty0 797--816, 1993.

\bibitem[Damarackas and Paulauskas(2014)]{Damarackas2014}
J.~Damarackas and V.~Paulauskas.
\newblock Properties of spectral covariance for linear processes with infinite
  variance.
\newblock \emph{Lith. Math. J.}, 54\penalty0 (3):\penalty0 252--276, 2014.

\bibitem[Damarackas and Paulauskas(2017)]{Damarackas2017}
J.~Damarackas and V.~Paulauskas.
\newblock Spectral covariance and limit theorems for random fields with
  infinite variance.
\newblock \emph{J. Multivariate Anal.}, 153:\penalty0 156--175, 2017.

\bibitem[Davison and Gholamrezaee(2012)]{davison-gholamrezaee-2012}
A.~C. Davison and M.~M. Gholamrezaee.
\newblock Geostatistics of extremes.
\newblock \emph{P. Roy. Soc. A-Math. Phy.}, 468\penalty0 (2138):\penalty0
  581--608, 2012.

\bibitem[Dehling and Philipp(2002)]{DehlPhil02}
H.~Dehling and W.~Philipp.
\newblock Empirical process techniques for dependent data.
\newblock In \emph{Empirical process techniques for dependent data}, pages
  3--113. Birkh\"{a}user Boston, Boston, MA, 2002.

\bibitem[Dombry and Kabluchko(2017)]{dombry-kabluchko-17}
C.~Dombry and Z.~Kabluchko.
\newblock Ergodic decompositions of stationary max-stable processes in terms of
  their spectral functions.
\newblock \emph{Stochastic Process. Appl.}, 127\penalty0 (6):\penalty0
  1763--1784, 2017.

\bibitem[Doukhan et~al.(2003)Doukhan, Oppenheim, and Taqqu]{Doukhan2003}
P.~Doukhan, G.~Oppenheim, and M.~S. Taqqu, editors.
\newblock \emph{Theory and applications of long-range dependence}.
\newblock Birkh\"{a}user Boston, Inc., Boston, MA, 2003.

\bibitem[Embrechts et~al.(1997)Embrechts, Kl{\"u}ppelberg, and Mikosch]{EKM97}
P.~Embrechts, C.~Kl{\"u}ppelberg, and T.~Mikosch.
\newblock \emph{Modelling Extremal Events for Insurance and Finance}.
\newblock Springer-Verlag, Berlin Heidelberg, 1997.

\bibitem[Giraitis et~al.(2012)Giraitis, Koul, and
  Surgailis]{GiraitisKoulSurg12}
L.~Giraitis, H.~L. Koul, and D.~Surgailis.
\newblock \emph{Large sample inference for long memory processes}.
\newblock Imperial College Press, London, 2012.

\bibitem[Heyde and Yang(1997)]{HeydeYang97}
C.~C. Heyde and Y.~Yang.
\newblock On defining long-range dependence.
\newblock \emph{J. Appl. Probab.}, 34\penalty0 (4):\penalty0 939--944, 1997.

\bibitem[Hosoya(1978)]{Char_Fct_Hosoya}
Y.~Hosoya.
\newblock Discrete-time stable processes and their certain properties.
\newblock \emph{Ann. Probab.}, 6\penalty0 (1):\penalty0 94--105, 02 1978.

\bibitem[Hurst(1951)]{hurst51}
H.~E. Hurst.
\newblock Long-term storage capacity of reservoirs.
\newblock \emph{Trans. Amer. Soc. Civil Eng.}, 116:\penalty0 770--799, 1951.

\bibitem[Jach et~al.(2012)Jach, McElroy, and Politis]{Jach2012}
A.~Jach, T.~McElroy, and D.~N. Politis.
\newblock Subsampling inference for the mean of heavy-tailed long-memory time
  series.
\newblock \emph{J. Time Series Anal.}, 33\penalty0 (1):\penalty0 96--111, 2012.

\bibitem[Kabluchko and Schlather(2010)]{kabluchko-schlather-10}
Z.~Kabluchko and M.~Schlather.
\newblock Ergodic properties of max-infinitely divisible processes.
\newblock \emph{Stochastic Process. Appl.}, 120\penalty0 (3):\penalty0
  281--295, 2010.

\bibitem[Kabluchko et~al.(2009)Kabluchko, Schlather, and de~Haan]{KSH09}
Z.~Kabluchko, M.~Schlather, and L.~de~Haan.
\newblock Stationary max-stable fields associated to negative definite
  functions.
\newblock \emph{Ann. Probab.}, 37\penalty0 (5):\penalty0 2042--2065, 2009.

\bibitem[Kampf et~al.(2020)Kampf, Shevchenko, and Spodarev]{Kampf2020}
J.~Kampf, G.~Shevchenko, and E.~Spodarev.
\newblock Nonparametric estimation of the kernel function of symmetric stable
  moving average random functions.
\newblock \emph{Ann. Inst. Statist. Math.}, to appear, 2020.
\newblock \doi{10.1007/510463-020-00751-6}.

\bibitem[Kasahara et~al.(1988)Kasahara, Maejima, and Vervaat]{Kasahara1988}
Y.~Kasahara, M.~Maejima, and W.~Vervaat.
\newblock Log-fractional stable processes.
\newblock \emph{Stochastic Process. Appl.}, 30\penalty0 (2):\penalty0 329--339,
  1988.

\bibitem[Kokoszka and Mikosch(1997)]{Kokoszka1997}
P.~Kokoszka and T.~Mikosch.
\newblock The integrated periodogram for long-memory processes with finite or
  infinite variance.
\newblock \emph{Stochastic Process. Appl.}, 66\penalty0 (1):\penalty0 55--78,
  1997.

\bibitem[Kokoszka and Taqqu(1996)]{Kokoszka1996}
P.~S. Kokoszka and M.~S. Taqqu.
\newblock Parameter estimation for infinite variance fractional {ARIMA}.
\newblock \emph{Ann. Statist.}, 24\penalty0 (5):\penalty0 1880--1913, 1996.

\bibitem[Koul and Surgailis(2018)]{Koul2018}
H.~L. Koul and D.~Surgailis.
\newblock Asymptotic distributions of some scale estimators in nonlinear models
  with long memory errors having infinite variance.
\newblock \emph{J. Time Series Anal.}, 39\penalty0 (3):\penalty0 273--298,
  2018.

\bibitem[{Kulik} and {Spodarev}(2019)]{Kulik_Spodarev_LRD}
R.~{Kulik} and E.~{Spodarev}.
\newblock {Long range dependence of heavy tailed random functions}.
\newblock \emph{arXiv e-prints}, page arXiv:1706.00742v3, 2019.

\bibitem[Lavancier(2006)]{Lavancier06}
F.~Lavancier.
\newblock Long memory random fields.
\newblock In \emph{Dependence in probability and statistics}, volume 187 of
  \emph{Lecture Notes in Statist.}, pages 195--220. Springer, New York, 2006.

\bibitem[Maejima and Yamamoto(2003)]{Maejima2003}
M.~Maejima and K.~Yamamoto.
\newblock Long-memory stable {O}rnstein-{U}hlenbeck processes.
\newblock \emph{Electron. J. Probab.}, 8\penalty0 (19):\penalty0 1--18, 2003.

\bibitem[Magdziarz and Weron(2007)]{Magdziarz2007}
M.~Magdziarz and A.~Weron.
\newblock Fractional {L}angevin equation with {$\alpha$}-stable noise. {A} link
  to fractional {ARIMA} time series.
\newblock \emph{Studia Math.}, 181\penalty0 (1):\penalty0 47--60, 2007.

\bibitem[McCulloch(1986)]{McCulloch86}
J.~H. McCulloch.
\newblock Simple consistent estimators of stable distribution parameters.
\newblock \emph{Commun. Stat-Simul C}, 15\penalty0 (4):\penalty0 1109--1136,
  1986.

\bibitem[Montillet and Yu(2015)]{Montillet2015}
J.-P. Montillet and K.~Yu.
\newblock Modeling geodetic processes with {L}evy {$\alpha$}-stable
  distribution and {FARIMA}.
\newblock \emph{Math. Geosci.}, 47\penalty0 (6):\penalty0 627--646, 2015.

\bibitem[Oesting et~al.(2017)Oesting, Schlather, and
  Friederichs]{oesting-friederichs-schlather-2017}
M.~Oesting, M.~Schlather, and P.~Friederichs.
\newblock Statistical post-processing of forecasts for extremes using bivariate
  {B}rown--{R}esnick processes with an application to wind gusts.
\newblock \emph{Extremes}, 20\penalty0 (2):\penalty0 309--332, 2017.

\bibitem[Panas(2001)]{panas2001estimating}
E.~Panas.
\newblock Estimating fractal dimension using stable distributions and exploring
  long memory through arfima models in athens stock exchange.
\newblock \emph{Appl Financ Econ}, 11\penalty0 (4):\penalty0 395--402, 2001.

\bibitem[Paulauskas(2016)]{Paul16}
V.~Paulauskas.
\newblock Some remarks on definitions of memory for stationary random processes
  and fields.
\newblock \emph{Lith. Math. J.}, 56\penalty0 (2):\penalty0 229--250, 2016.

\bibitem[Paulauskas(1976)]{Paulauskas1976}
V.~J. Paulauskas.
\newblock Some remarks on multivariate stable distributions.
\newblock \emph{J. Multivariate Anal.}, 6\penalty0 (3):\penalty0 356--368,
  1976.

\bibitem[Pilipauskait\.{e} and Surgailis(2016)]{Pilip2016}
V.~Pilipauskait\.{e} and D.~Surgailis.
\newblock Anisotropic scaling of the random grain model with application to
  network traffic.
\newblock \emph{J. Appl. Probab.}, 53\penalty0 (3):\penalty0 857--879, 2016.

\bibitem[Rachev and Samorodnitsky(2002)]{Rachev2002}
S.~T. Rachev and G.~Samorodnitsky.
\newblock Erratum to: ``{L}ong strange segments in a long-range-dependent
  moving average'' [{S}tochastic {P}rocess. {A}ppl. {\bf 93} (2001), no. 1,
  119--148;].
\newblock \emph{Stochastic Process. Appl.}, 101\penalty0 (1):\penalty0 161,
  2002.

\bibitem[Reed and Simon(1981)]{Simon_Reed}
M.~Reed and B.~Simon.
\newblock \emph{Methods of Modern Mathematical Physics, Volume 1, Functional
  Analysis.}
\newblock Academic Press, 1981.

\bibitem[Resnick(2007)]{resnick-2007}
S.~I. Resnick.
\newblock \emph{Heavy-Tail Phenomena: Probabilistic and Statistical Modeling}.
\newblock Springer, New York, 2007.

\bibitem[Resnick(2013)]{resnick-87}
S.~I. Resnick.
\newblock \emph{Extreme Values, Regular Variation and Point Processes}.
\newblock Springer, 2013.

\bibitem[Samorodnitsky(2004)]{Samorod04}
G.~Samorodnitsky.
\newblock Extreme value theory, ergodic theory and the boundary between short
  memory and long memory for stationary stable processes.
\newblock \emph{Ann. Probab.}, 32\penalty0 (2):\penalty0 1438--1468, 2004.

\bibitem[Samorodnitsky(2016)]{Samorod16}
G.~Samorodnitsky.
\newblock \emph{Stochastic processes and long range dependence}.
\newblock Springer Series in Operations Research and Financial Engineering.
  Springer, Cham, 2016.

\bibitem[Samorodnitsky and Taqqu(1994)]{SamorodTaqqu94}
G.~Samorodnitsky and M.~S. Taqqu.
\newblock \emph{Stable non-{G}aussian random processes}.
\newblock Stochastic Modeling. Chapman \& Hall, New York, 1994.
\newblock Stochastic models with infinite variance.

\bibitem[Schilling(2017)]{Schilling_measure_theory}
R.~L. Schilling.
\newblock \emph{Measures, Integrals and Martingales}.
\newblock Cambridge University Press, second edition, 2017.

\bibitem[Schlather(2002)]{schlather02}
M.~Schlather.
\newblock Models for stationary max-stable random fields.
\newblock \emph{Extremes}, 5\penalty0 (1):\penalty0 33--44, 2002.

\bibitem[Schlather and Tawn(2003)]{schlather-tawn-03}
M.~Schlather and J.~A. Tawn.
\newblock A dependence measure for multivariate and spatial extreme values:
  Properties and inference.
\newblock \emph{Biometrika}, 90\penalty0 (1):\penalty0 139--156, 2003.

\bibitem[Stoev(2008)]{stoev-08}
S.~A. Stoev.
\newblock On the ergodicity and mixing of max-stable processes.
\newblock \emph{Stochastic Process. Appl.}, 118\penalty0 (9):\penalty0
  1679--1705, 2008.

\bibitem[Strokorb and Schlather(2015)]{strokorb-schlather-15}
K.~Strokorb and M.~Schlather.
\newblock An exceptional max-stable process fully parameterized by its extremal
  coefficients.
\newblock \emph{Bernoulli}, 21\penalty0 (1):\penalty0 276--302, 2015.

\bibitem[Wang and Stoev(2010)]{wang-stoev-10}
Y.~Wang and S.~Stoev.
\newblock On the structure and representations of max-stable processes.
\newblock \emph{Adv. Appl. Probab.}, 42\penalty0 (3):\penalty0 855--877, 2010.

\bibitem[Zhang and Smith(2010)]{zhang-smith-2010}
Z.~Zhang and R.~L. Smith.
\newblock On the estimation and application of max-stable processes.
\newblock \emph{J. Statist. Plann. Inference}, 140\penalty0 (5):\penalty0
  1135--1153, 2010.

\end{thebibliography}
	
	\appendix

	
	\section{Appendix}
	\label{sec: appendix}

	\begin{Lemma}[]
		\label{lem: Schilling_lemma}
		Suppose $U$ and $V$ are identically distributed random variables with marginal characteristic function $\CFU{}$ and joint characteristic function  $\CFV{}{}$. Also, let $U^{\prime}$ and $V^{\prime}$ be independent copies of $U$ and $V$, respectively. Then, for a Fourier transform of a finite measure $\mu$ denoted by $\psi: \mathbb{R} \rightarrow \mathbb{C}, \psi(s) = \int_{\mathbb{R}} \exp \{ isx \}\, \mu(dx)$ it holds that
		\begin{align*}
			\mathbb{E} \int_{\mathbb{R}} \int_{\mathbb{R}} \widehat{f}_a(s_1, s_2) \overline{\psi(s_1)\psi(s_2)}\  ds_1\, ds_2 =
			 \int_{\mathbb{R}} \int_{\mathbb{R}} \mathbb{E} \widehat{f}_a(s_1, s_2) \overline{\psi(s_1)\psi(s_2)}\  ds_1\, ds_2,
		\end{align*}
		where $\widehat{f}_a(s_1, s_2) = \int_{\RR} \int_{\mathbb{R}} e^{i(s_1u+s_2v)} f_a(u,v)\ du\, dv$, $a > 0$ and 
		\begin{align*}
		f_a(u,v) = \mathds{1} \big\{ U > u > -a, V > v> -a \big\} - \mathds{1} \big\{ U^{\prime} > u > -a, V^{\prime} > v > -a \big\}.
		\end{align*}
		
		\begin{proof}
		Let us define $g_a(s_1, s_2) = \widehat{f}_a(s_1, s_2) \mathds{1} \{ s_1, s_2 \in (-r, r) \}$ for $r > 0$ and denote 
			\begin{align*}
				a_r = \int_{\RR} \int_{\RR} g_a(s_1, s_2)  \overline{\psi(s_1)\psi(s_2)} \ ds_1\, ds_2 = \int_{\RR} \int_{\RR} \widehat{g}_a(s_1, s_2)  \ \mu(s_1)\, \mu(s_2),
			\end{align*} 
			where we have used \cite[Theorem 19.12]{Schilling_measure_theory} in the last equality and $\widehat{g}_a$ is the inverse Fourier transform of $g_a$. It can be computed by tedious, yet simple calculations as
			\begin{align}
				\widehat{g}_a(s_1, s_2) 
				&= \bigg( \int_{-r}^{r} \frac{e^{-is_1x
				}}{ix} \Big( e^{ixU} - e^{-ixa} \Big) \ dx \bigg) \bigg( \int_{-r}^{r} \frac{e^{-is_2y
			}}{iy} \Big( e^{iyV} - e^{-iya} \Big) \ dy \bigg) \notag \\
		&\quad -
		\bigg( \int_{-r}^{r} \frac{e^{-is_1x
		}}{ix} \Big( e^{ixU^{\prime}} - e^{-ixa} \Big) \ dx \bigg) \bigg( \int_{-r}^{r} \frac{e^{-is_2y
		}}{iy} \Big( e^{iyV^{\prime}} - e^{-iya} \Big) \ dy \bigg).
\label{eq: g_Schillingbeweis}
			\end{align}	
			Computing one of these factors gives us
			\begin{align*}
				 \int_{-r}^{r} \frac{e^{-is_1x
				}}{ix} \Big( e^{ixU} - e^{-ixa} \Big) \ dx &= \int_{-r}^{r} \int_{-a}^{U} e^{-is_1x} e^{ixt}\ dt\, dx = \int_{-a}^{U} \int_{-r}^{r}e^{ix(t-s_1)} \ dx\, dt\\
				&=
				2 \int_{r(U - s_1)}^{r(a-s_1)} \frac{1}{t} \, \frac{e^{it} - e^{-it}}{2i}\ dt = 2 \int_{r(U - s_1)}^{r(a-s_1)} \frac{\sin t}{t}\ dt < C
			\end{align*}
			for some finite constant $C > 0$ which is independent of $U$, $V$ and $r$. This constant exists because the function $ \text{Si}(x) = \int_{0}^{x} \frac{\sin t}{t}\ dt$ is bounded. Thus,  $\vert \widehat{g}_a(s_1, s_2) \vert \leq 2 C^2$ for all $s_1, s_2 \in \mathbb{R}$ by equality \eqref{eq: g_Schillingbeweis} and $\vert a_r \vert \leq 2 C^2 \mu^2(\mathbb{R})$. Therefore, we can apply the dominated convergence theorem and $\mathbb{E}[\lim_{r \to \infty} a_r] = \lim_{r \to \infty} \mathbb{E}[a_r]$. 
		\end{proof}
	\end{Lemma}
	
	\begin{Lemma}[]
		\label{lem: inequalities_bundled}
		
		\begin{enumerate}[(a)]
			\item \label{lem: Vitalii inequality}
			Suppose $a, b \in \mathbb{R}$ and $\alpha \in (0,2)$. Then it holds that
			\begin{align*}
			\Big\vert \vert a  \vert^{\alpha} + \vert b \vert^{\alpha} - \vert a-b \vert^{\alpha}  \Big\vert \leq 2 \vert a \vert ^{\alpha/2} \vert b \vert ^{\alpha/2}.
			\end{align*}

			\item \label{lem: Ungleichen groesser, kleiner null}
			Suppose $a,b \geq 0$ and $\alpha \geq 0$, then $ \vert a - b \vert^\alpha \leq a^\alpha + b^\alpha$.

				
			\item \label{lem: inequalites LRD2} 
			Let $\alpha>0,a\geq 0,b\geq 0,$ then
			\begin{equation}
			\label{V:inequalites LRD} \int_0^1\int_0^1 \frac{(a s_1+b s_2)^{\alpha}-|a s_1-b s_2|^{\alpha}}{s_1 s_2}\dd s_1 \dd s_2 \geq  C_\alpha \left(a^{\alpha}\wedge b^{\alpha}\right),
			\end{equation}
			where $C_\alpha=\frac{2}{\alpha}\int_0^1 \frac{(1+u)^{\alpha}-(1-u)^{\alpha}}{u}\dd u\geq \frac{4(2^\alpha-1)}{\alpha(\alpha+1)}.$
		\end{enumerate}
	\end{Lemma}		
	\begin{proof}
		\begin{enumerate}[(a)]
			\item Recall that the covariance function of a fractional Brownian motion $B^{H}=\{ B_t^{H}, t \in \mathbb{R} \}$ with Hurst index $H \in (0,1)$ equals
			\begin{align}
			\mathbb{E}[B_t^{H}B_s^{H}] = \frac{1}{2}( \vert t \vert ^{2H} + \vert s \vert ^{2H} - \vert t-s \vert ^{2H}), \label{eq: covariance FBM}
			\end{align} 
			where $s,t \in \mathbb{R}$. Using the Cauchy-Schwarz inequality we have that
			\begin{align}
			\vert \mathbb{E}[B_t^{H}B_s^{H}] \vert^2 \leq \mathbb{E}[(B_t^{H})^2]\mathbb{E}[(B_s^{H})^2] = \vert t \vert ^{2H}  \vert s \vert ^{2H}. \label{eq: FBM Cauchy-Schwarz}
			\end{align}
			Denoting $\alpha=2H$ and combining (\ref{eq: covariance FBM}) and (\ref{eq: FBM Cauchy-Schwarz}) we get
			$
			\bigg\vert\frac{1}{2}( \vert t  \vert^{\alpha} + \vert s \vert^{\alpha} - \vert t-s \vert^{\alpha})  \bigg\vert^2 \leq   \vert t \vert ^{\alpha}  \vert s \vert ^{\alpha},
			$
			or, equivalently,
			$\bigg\vert \vert t  \vert^{\alpha} + \vert s \vert^{\alpha} - \vert t-s \vert^{\alpha}  \bigg\vert \leq  2 \vert t \vert ^{\alpha/2}  \vert s \vert ^{\alpha/2}.
			$
			
			\item Suppose $a < b$, then 
			\begin{align*}
			a^\alpha + b^\alpha - \vert a - b \vert^\alpha = a^\alpha + b^\alpha - ( b - a)^\alpha \geq a^\alpha + b^\alpha - b^\alpha = a^\alpha \geq 0.
			\end{align*}
			The claim follows analogously if $a \geq b$.

		\item For simplicity, consider the case $0<a<b.$ After the change of variable, the integral in \eqref{V:inequalites LRD} rewrites
		\begin{align}
		\nonumber&\int_0^a\int_0^b \frac{(z_2+z_1)^{\alpha}-|z_1-z_2|^{\alpha}}{z_1 z_2}\dd z_1 \dd z_2=2\iint_{0\leq z_1\leq z_2\leq a}\frac{(z_2+z_1)^{\alpha}-(z_2-z_1)^{\alpha}}{z_1 z_2}\dd z_1 \dd z_2\\
		\nonumber&+\int_0^a \int_a^b\frac{(z_2+z_1)^{\alpha}-(z_2-z_1)^{\alpha}}{z_1 z_2}\dd z_2 \dd z_1\geq 
		2\int_0^a \left[\int_0^{z_2} \frac{(z_2+z_1)^{\alpha}-(z_2-z_1)^{\alpha}}{z_1}\dd z_1\right] \frac{\dd z_2}{ z_2}\\
		&=2\int_0^a z_2^{\alpha-1} \dd z_2\left[\int_0^1 \frac{(1+u)^{\alpha}-(1-u)^{\alpha}}{u}\dd u\right]=C_\alpha a^\alpha.
		\end{align}
		\end{enumerate}	
	\end{proof}

	\begin{Lemma}
		\label{lemma: ineq}
		Let $X = \{ X(t), \ t \in \mathbb{R} \}$ be a S$\alpha$S moving average process with parameter $\alpha \in (0, 2)$, nonnegative kernel function $m \in \lp{\alpha}{\mathbb{R}}$, $m(x) >0$ on a set of positive Lebesgue measure and $\alpha-$spectral covariance $\rho_t.$ Let $\varphi$ and $\varphi_t$, $t \in \mathbb{R}$, be the characteristic functions given in equations \eqref{eq: char_fct_univ} and \eqref{eq: char_fct_biv}. Then, 
		\begin{align}
		I_1=\int_{\RR_+} \int_{\RR_+} \frac{\vert\CFB{+}{-} - \CFA{s_1}\CFA{s_2}\vert}{s_1s_2} \, \dd s_1 \, \dd s_2 &\leq \frac{8\pi}{\alpha^2 \Vert m \Vert^{2\alpha}_{\alpha}}  \frac{\rho_t}{\sqrt{\Vert m \Vert^{2\alpha}_{\alpha}-\rho_t^2}} , \label{eq: I_1 finite} \\
		I_2=\int_{\RR_+} \int_{\RR_+} \frac{\vert\CFB{+}{+} - \CFA{s_1}\CFA{s_2}\vert}{s_1s_2} \, \dd s_1 \, \dd s_2&\leq
		\frac{8\pi}{\alpha^2 \Vert m \Vert^{2\alpha}_{\alpha}}  \frac{\rho_t}{\sqrt{\Vert m \Vert^{2\alpha}_{\alpha}-\rho_t^2}}.
		\label{eq: I_2 finite} 
		\end{align}
	\end{Lemma}
	
	\begin{proof}
		First, compute for $\alpha \in (0,2)$ the absolute value of the difference of characteristic functions in $I_1$ for any $s_1, s_2 > 0$:
		\begin{align}
		&\vert\CFB{+}{-} - \CFA{s_1}\CFA{s_2}\vert 
		=  \CFB{+}{-} \ \cdot \bigg\vert 1 - \frac{\CFA{s_1}\CFA{s_2}}{\CFB{+}{-}} \bigg\vert\notag \\[2mm]
		&= \CFB{+}{-}  
		\cdot \bigg\vert 1 -  \exp \bigg\{- \int_{\mathbb{R}} \bigg( \underbrace{\vert s_1m(-x)\vert^{\alpha} + \vert s_2m(t-x)\vert^{\alpha}  -\vert s_1m(-x)-s_2m(t-x)\vert^{\alpha}}_{\geq 0 \text{ by \Cref{lem: inequalities_bundled}(\ref{lem: Ungleichen groesser, kleiner null}) }} \bigg) \, \dd x \bigg\} \bigg\vert\notag \\[2mm]
		&
		\leq  \CFB{+}{-}  \cdot \int_{\mathbb{R}} \bigg( \vert s_1m(-x)\vert^{\alpha} + \vert s_2m(t-x)\vert^{\alpha}  -\vert s_1m(-x)-s_2m(t-x)\vert^{\alpha} \bigg) \, \dd x\notag \\[1mm]
		&\hspace{8cm} \Big( \text{by using } \vert e^{-x}-1 \vert \leq  x  \text{ for all } x \geq 0 \Big)  \notag\\[1mm]
		&
		\leq \CFB{+}{-}  \cdot 2s_1^{\alpha/2}  s_2^{\alpha/2} \int_{\mathbb{R}} m^{\alpha/2}(-x)m^{\alpha/2}(t-x)  \, \dd x, \label{eq: estimate_1}
		\end{align}
		where we have used \Cref{lem: inequalities_bundled}(\ref{lem: Vitalii inequality}) in the last inequality. \\
		Similarly, we compute $\vert\CFB{+}{+} - \CFA{s_1}\CFA{s_2}\vert$ for the case $\alpha\in (1,2)$ as
		\begin{align}
		&\vert\CFB{+}{+} - \CFA{s_1}\CFA{s_2}\vert\notag \\[2mm]
		&= \CFA{s_1}\CFA{s_2}  \cdot \bigg\vert \frac{\CFB{+}{+}}{\CFA{s_1}\CFA{s_2}}  -1 \bigg\vert \notag\\[2mm]
		&= \CFA{s_1}\CFA{s_2}  \cdot \bigg\vert   \exp \bigg\{- \int_{\mathbb{R}} \bigg( \underbrace{\vert s_1m(-x)+s_2m(t-x)\vert^{\alpha}  - \vert s_1m(-x)\vert^{\alpha} - \vert s_2m(t-x)\vert^{\alpha}}_{\geq 0 \text{ since } \alpha > 1} \bigg) \, \dd x \bigg\} -1 \bigg\vert\notag \\[2mm] 
		&\leq  \CFA{s_1}\CFA{s_2}  \cdot \int_{\mathbb{R}} \bigg( \vert s_1m(-x)+s_2m(t-x)\vert^{\alpha}  - \vert s_1m(-x)\vert^{\alpha} - \vert s_2m(t-x)\vert^{\alpha} \bigg) \, \dd x\notag \\[2mm]
		&\leq  \CFA{s_1}\CFA{s_2} \cdot 2s_1^{\alpha/2}  s_2^{\alpha/2} \int_{\mathbb{R}} m^{\alpha/2}(-x)m^{\alpha/2}(t-x) \, \dd x\notag\\ 
		&\leq \CFB{+}{-} \cdot 2s_1^{\alpha/2}  s_2^{\alpha/2} \int_{\mathbb{R}} m^{\alpha/2}(-x)m^{\alpha/2}(t-x)  \, \dd x, \label{eq: estimate_2}
		\end{align}
		where, again, we have used \Cref{lem: inequalities_bundled}(\ref{lem: Vitalii inequality}) in the last inequality. 
		Using the same arguments, we get for the case $\alpha\in (0,1)$ that
		\begin{align}
		&\vert\CFB{+}{+} - \CFA{s_1}\CFA{s_2}\vert=  \CFB{+}{+} \ \cdot \bigg\vert 1 - \frac{\CFA{s_1}\CFA{s_2}}{\CFB{+}{+}} \bigg\vert\notag \\[2mm]
		&= \CFB{+}{+}  
		\cdot \bigg\vert 1 -  \exp \bigg\{- \int_{\mathbb{R}} \bigg( \underbrace{\vert s_1m(-x)\vert^{\alpha} + \vert s_2m(t-x)\vert^{\alpha}  -\vert s_1m(-x)+s_2m(t-x)\vert^{\alpha}}_{\geq 0 \text{ since } \alpha < 1} \bigg) \, \dd x \bigg\} \bigg\vert\notag\\
		&\leq  \CFB{+}{+} \cdot 2s_1^{\alpha/2}  s_2^{\alpha/2} \int_{\mathbb{R}} m^{\alpha/2}(-x)m^{\alpha/2}(t-x) \, \dd x\notag\\
		&\leq \CFB{+}{-} \cdot 2s_1^{\alpha/2}  s_2^{\alpha/2} \int_{\mathbb{R}} m^{\alpha/2}(-x)m^{\alpha/2}(t-x)  \, \dd x, \label{eq: estimate_21}
		\end{align}
		Now, plugging the estimates \eqref{eq: estimate_1},\eqref{eq: estimate_2} and \eqref{eq: estimate_21} into $I_1$ and $I_2$ we get
		\begin{align}
		\nonumber  I_1,I_2 &\leq 2  \bigg( \int_{\RR_+} \int_{\RR_+}\frac{\CFB{+}{-}}{s_1^{1-\alpha/2}s_2^{1-\alpha/2}} \, \dd s_1 \, \dd s_2 \bigg) \bigg( \int_{\mathbb{R}} m^{\alpha/2}(-x)m^{\alpha/2}(t-x)  \, \dd x \bigg)\,
		\\[2mm] 
		\label{ineq: I1} &= 2  \bigg( \int_{\RR_+} \int_{\RR_+}\frac{\CFB{+}{-}}{s_1^{1-\alpha/2}s_2^{1-\alpha/2}} \, \dd s_1 \, \dd s_2 \bigg) \rho_t. 
		\end{align}
		
		We estimate the integral in \eqref{ineq: I1} 
		from above via the density of a bivariate normal law. Thus,
		\begin{align}
		\int_{\RR_+} &\int_{\RR_+}\frac{\CFB{+}{-}}{s_1^{1-\alpha/2}s_2^{1-\alpha/2}} \, \dd s_1 \, \dd s_2  \notag \\[2mm]
		&= \int_{\RR_+} \int_{\RR_+} \bigg( \frac{1}{s_1 s_2} \bigg)^{1-\tfrac{\alpha}{2}} \exp \bigg\{ -  \int_{\mathbb{R}} \vert s_1m(-x)+(-s_2)m(t-x)\vert^{\alpha} \, \dd x \bigg\}  \dd s_1 \, \dd s_2 \notag \\[2mm] 
		&= \int_{\RR_+} \int_{\RR_+} \bigg( \frac{1}{s_1 s_2} \bigg)^{1-\tfrac{\alpha}{2}}\exp \bigg\{-\int_{\mathbb{R}} \vert s_1m(-x)\vert^{\alpha} \, \dd x - \int_{\mathbb{R}} \vert (-s_2)m(t-x)\vert^{\alpha} \, \dd x \notag \\[2mm]
		& +  \int_{\mathbb{R}} \vert s_1m(-x)\vert^{\alpha} + \vert (-s_2)m(t-x)\vert^{\alpha}  -\vert s_1m(-x)+(-s_2)m(t-x)\vert^{\alpha} \, \dd x \bigg\} \dd s_1 \, \dd s_2 \notag \\[2mm]
		&\leq  \int_{\RR_+} \int_{\RR_+} \bigg( \frac{1}{s_1 s_2} \bigg)^{1-\tfrac{\alpha}{2}} \exp \bigg\{-  s_1^{\alpha}\int_{\mathbb{R}}  m^{\alpha}(x) \, \dd x - s_2^{\alpha} \int_{\mathbb{R}}  m^{\alpha}(x) \, \dd x \notag \\[2mm]
		& \hspace{2cm}+  2 (s_1 s_2)^{\alpha/2}\int_{\mathbb{R}}  m^{\alpha/2}(-x)m^{\alpha/2}(t-x)\, \dd x \bigg\} \dd s_1 \, \dd s_2 \notag\\[2mm]
		&= \int_{\RR_+} \int_{\RR_+} \bigg( \frac{1}{s_1 s_2} \bigg)^{1-\tfrac{\alpha}{2}} \exp \bigg\{ -(s_1^{\alpha/2} \sigma)^2 - (s_2^{\alpha/2} \sigma)^2 + 2 \tilde{\rho_t} (\sigma s_1^{\alpha/2}) (\sigma s_2^{\alpha/2}) \bigg\} \, \dd s_1 \, ds_2, \label{eq: I_1 before substitution}
		\end{align}
		where we denoted
		\begin{align}
		\sigma^2= \rho_0=\int_{\mathbb{R}} m^{\alpha} (x)\, \dd x = \Vert m \Vert_{\alpha}^{\alpha} \quad \text{ and } \quad \tilde{\rho_t}=\frac{\rho_t}{\sigma^2}= \frac{1}{\sigma^2}\int_{\mathbb{R}} \vert m(-x)m(t-x)\vert^{\alpha/2} \, \dd x \label{eq: sigma_and_rho_sigma defined}
		\end{align}
		in the last equality.
		Now,  we have a substitution
		\begin{align*}
		\frac{y_i}{\sqrt{2 (1-\tilde{\rho_t}^2)}} = \sigma s_i^{\alpha/2},  \text{ or }  s_i = \bigg( \frac{y_i}{\sigma\sqrt{2 (1-\tilde{\rho_t}^2)}} \bigg)^{2/\alpha},\quad i=1, 2,
		\end{align*}
		so that
		\begin{equation*}
		ds_i =  \frac{\frac{2}{\alpha}y_i^{\tfrac{2}{\alpha}-1}}{\big(\sigma\sqrt{2 (1-\tilde{\rho_t}^2)}\,\big)^{2/\alpha}} \ dy_i,\quad i=1,2. 
		\end{equation*}
		Then, relation \eqref{eq: I_1 before substitution} rewrites
		\begin{align*}
		&\int_{\RR_+} \int_{\RR_+}\bigg( \frac{\big(2\sigma^2 (1-\tilde{\rho_t}^2)\,\big)^{2/\alpha}}{y_1^{2/\alpha}y_2^{2/\alpha}} \bigg)^{1-\tfrac{\alpha}{2}} 
		\, \exp \bigg\{ - \frac{y_1^2 + y_2^2 - 2 \tilde{\rho_t} y_1 y_2}{2(1-\tilde{\rho_t}^2)} \bigg\} 
		\,  \frac{( \frac{2}{\alpha} )^{^2}\, y_1^{\tfrac{2}{\alpha}-1} y_2^{\tfrac{2}{\alpha}-1}}{\big(2\sigma^2 (1-\tilde{\rho_t}^2)\,\big)^{2/\alpha}}\ dy_1 \, dy_2 \\[2mm] 
		&= \big(2\sigma^2 (1-\tilde{\rho_t}^2)\,\big)^{-1} \bigg( \frac{2}{\alpha} \bigg)^2 \int_{\RR_+} \int_{\RR_+} \exp \bigg\{ - \frac{y_1^2 + y_2^2 - 2 \tilde{\rho_t} y_1 y_2}{2(1-\tilde{\rho_t}^2)} \bigg\}\ dy_1 \, dy_2 \\[2mm]
		&= \frac{4\pi}{\alpha^2\sigma^2}\big(1-\tilde{\rho_t}^2\,\big)^{-1/2}   \underbrace{\int_{0}^{\infty} \int_{0}^{\infty} \frac{1}{2\pi \sqrt{1-\tilde{\rho_t}^2}} \exp \bigg\{ - \frac{y_1^2 + y_2^2 - 2 \tilde{\rho_t} y_1 y_2}{2(1-\tilde{\rho_t}^2)} \bigg\}\, \dd s_1 \, ds_2}_{\leq 1 \text{ as a density of bivariate normal distribution} } \\[2mm]
		&\leq \frac{4\pi}{\alpha^2\sigma^2} \big(1-\tilde{\rho_t}^2\,\big)^{-1/2} = \frac{4\pi}{\alpha^2 \Vert m \Vert^{2\alpha}_{\alpha}} \big(\rho_0^2-\rho_t^2\,\big)^{-1/2} < \infty \text{ for } \rho_t \neq \rho_0.
		\end{align*}
		 Now,  show that $\big(\rho_0^2-\rho_t^2\,\big)^{-1/2}$ is only infinite on a null set. More specifically, we show that $\rho_t = \rho_0$ if and only if $t = 0$. Recall that the Cauchy-Schwarz inequality (cf.~\cite{Simon_Reed}, Theorem S.3.) states that
		\begin{align*}
		\rho_t = \int_{\mathbb{R}} m^{\alpha/2}(-x)m^{\alpha/2}(t-x)\, \dd x \leq \int_{\mathbb{R}} m^{\alpha}(-x)\, \dd x = \rho_0,
		\end{align*}
		where equality holds if and only if there exists $\lambda_t\in \mathbb{R}$ such that $m^{\alpha/2}(-x)=\lambda_t m^{\alpha/2}(t-x) \text{ a.e}$. In this case, relation $\rho_t  = \rho_0$ yields $\lambda_t =1$.
		Note that due to $m$ being nonnegative, we can rewrite the condition as $m(-x) = m(t-x) \text{ a.e.}$ or  $m(x)= m(x+t) \text{ a.e.}$; hence, $m$ is a $t$-periodic function with $m(x) > 0$ on a set of positive Lebesgue measure which contradicts $m \in \lp{\alpha}{\mathbb{R}}$ because in that case $m(x) \rightarrow 0$ as $x \rightarrow \pm \infty$. Consequently, $\rho_t = \rho_0$ if and only if $t = 0$.

	\end{proof}
	
\textit{Proof of \Cref{thm: LRD_min_m}:}
 Let us choose $\mu= \delta_0$ where $\delta_0$ is the Dirac measure concentrated at zero. Obviously this measure is finite and by \Cref{Corr: SRD and LRD symmetric} we get for $t \in \mathbb{R}$:
		\begin{align}
		&\int_{\mathbb{R}} \int_{\mathbb{R}} \Cov (\mathds{1} \{ X(0) > u \}, \mathds{1} \{ X(t) > v \})\, \mu(\dd u)\, \mu(\dd v) \notag \\[2mm]
		&= \frac{1}{2 \pi^2} \int_{\RR_+} \int_{\RR_+} \frac{1}{s_1s_2}  \bigg( \CFB{+}{-}  - \CFB{+}{+} \bigg) \, \dd s_1\, \dd s_2 \label{eq: integral for LRD}.
		\end{align}
		We denote $A= \vert s_1m(-x) + s_2m(t-x) \vert ^{\alpha}$ and $B= \vert s_1m(-x) - s_2m(t-x) \vert ^{\alpha}$. Then, by $e^{x}-1 \geq x$ for all $x \in \mathbb{R}$ we estimate 
		\begin{align}
		\CFB{+}{-}  - \CFB{+}{+} &= \exp\bigg\{ - \int_{\mathbb{R}} B\, \dd x \bigg\} - \exp\bigg\{ - \int_{\mathbb{R}} A\, \dd x \bigg\}\notag\\[2mm]
		&= \exp\bigg\{ - \int_{\mathbb{R}} A\, \dd x \bigg\} \bigg( \exp\bigg\{ \int_{\mathbb{R}} [A- B]\, \dd x \bigg\} -1 \bigg)\notag\\[2mm]
		&\geq \exp\bigg\{ - \int_{\mathbb{R}} A\, \dd x \bigg\} \int_{\mathbb{R}} [A- B]\, \dd x \label{eq: first bound LRD phi difference}.
		\end{align}
		
		Thus, we obtain a lower bound for the right hand side of \eqref{eq: first bound LRD phi difference}:
		\begin{align*}
		&\exp \bigg\{ - \int_{\mathbb{R}}
		\vert s_1m(-x) + s_2m(t-x) \vert^{\alpha}\, \dd x \bigg\} \\
		&\hspace{2cm} \times \int_{\mathbb{R}}\left[ (s_1m(-x) + s_2m(t-x))^{\alpha} - \vert s_1m(-x) - s_2m(t-x) \vert^{\alpha}\right]\, \dd x.
		\end{align*}
		
		Note that for any $a,b>0,$ $(a+b)^{\alpha}=((a+b)^2)^{\alpha/2}\leq 2^{\alpha/2}(a^2+b^2)^{\alpha/2}\leq 2^{\alpha/2}(a^\alpha+b^\alpha).$  Thus, for $s_1,s_2\in[0,1]$
		\begin{align*}\exp \bigg\{ - \int_{\mathbb{R}}
		\vert s_1m(-x) + s_2m(t-x) \vert^{\alpha}\, \dd x \bigg\}\geq \exp \bigg\{ - (s_1^\alpha+s_2^{\alpha})2^{\frac{\alpha}{2}}
		\| m\|^\alpha_\alpha \bigg\}\geq e^{-4\|m\|^\alpha_\alpha}.
		\end{align*}
		Consequently, it holds
		\begin{align*}
		 &\int_{0}^{\infty} \int_{0}^{\infty} \frac{1}{s_1s_2}  \bigg( \CFB{+}{-}  - \CFB{+}{+} \bigg) \, \dd s_1\, \dd s_2\geq \int_{0}^{1} \int_{0}^{1} \frac{1}{s_1s_2}  \bigg( \CFB{+}{-}  - \CFB{+}{+} \bigg) \, \dd s_1\, \dd s_2 \notag\\[2mm] 
		\geq  &\int_{0}^{1} \int_{0}^{1} \frac{1}{s_1s_2}  
		\bigg( \exp \bigg\{ -\int_{\mathbb{R}} (s_1m(-x) + s_2m(t-x))^\alpha \, \dd x \bigg\}\notag \\[2mm]  
		& \times \int_{\mathbb{R}} \Big( (s_1m(-x) + s_2m(t-x))^\alpha-|s_1m(-x) - s_2m(t-x)|^\alpha\Big)\, \dd x \bigg) \, \dd s_1\, \dd s_2 \notag \\[2mm]
		\geq  &e^{-4\|m\|^\alpha_\alpha}\int_{0}^{1} \int_{0}^{1}\int_{\RR}\frac{ (s_1m(-x) + s_2m(t-x))^\alpha-|s_1m(-x) - s_2m(t-x)|^\alpha}{s_1s_2}  \, \dd x \,\dd s_1\, \dd s_2. 
		\end{align*}
		Now, by Fubini's theorem and \Cref{lem: inequalities_bundled}  (\ref{lem: inequalites LRD2}),   this is greater or equal to
		\begin{equation*}
		 e^{-4\|m\|^\alpha_\alpha}C_\alpha \int_{\mathbb{R}}  m^{\alpha}(-x) \wedge m^{\alpha}(t-x) \, \dd x.
		\end{equation*}

		Thus, for $\mu = \delta_0$ we have that
		\begin{align*}
		\int_{\mathbb{R}} \int_{\mathbb{R}} \Cov (\mathds{1} \{ X(0) > u \}, \mathds{1} \{ X(t) > v \})\, \mu(\dd u)\, \mu(\dd v) 
		\geq c \int_{\mathbb{R}} \big( m^\alpha(-x) \wedge m^\alpha(t-x) \big) \, \dd x. 
		\end{align*}
		
		Consequently, by \Cref{def: LRD_Spodarev} and Fubini's theorem, $X$ is long range dependent if
		\begin{align*}
		\int_{\mathbb{R}} \int_{\mathbb{R}} \big( m^\alpha(-x) \wedge m^\alpha(t-x) \big) \, \dd x \, \dd t &=\int_{\mathbb{R}} \int_{\mathbb{R}} \big( m^\alpha(-x) \wedge m^\alpha(t-x) \big) \, \dd t \, \dd x \\[2mm]
		&=\int_{\mathbb{R}} \int_{\mathbb{R}} \big( m^\alpha(x) \wedge m^\alpha(t) \big) \, \dd t \, \dd x = \infty. 
		\end{align*}

\end{document}